\documentclass{amsart}
\usepackage[T1]{fontenc}
\usepackage{amssymb}
\usepackage{amsmath}

\newtheorem{thm}{Theorem}
\newtheorem{lem}{Lemma}

\newtheorem{prop}{Proposition}
\newtheorem{rem}{Remark}

\usepackage{cite}
\usepackage[msc-links, initials]{amsrefs}

\usepackage{hyperref}
\hypersetup{
  colorlinks   = true,
  urlcolor     = blue,
  linkcolor    = blue,
  citecolor    = blue
}

\usepackage[capitalise]{cleveref}
\crefname{thm}{Thm.}{}
\crefname{prop}{Prop.}{}
\crefname{lem}{Lem.}{}
\crefname{rem}{Rem.}{}

\crefname{cor}{Cor.}{}
\crefname{prob}{Problem}{}
\crefname{figure}{Fig.}{}
\crefname{equation}{Eq.}{}
\crefname{exa}{Exa.}{}
\crefname{defi}{Def.}{}

\newcommand{\F}{\mathbb{F}}       
\newcommand\norm[1]{\Vert#1\Vert}
\newcommand{\abs}[1]{\left\lvert\mspace{1mu}#1\mspace{1mu}\right\rvert}
\newcommand{\Q}{\mathbb Q}
\newcommand{\C}{\mathbb C}
\newcommand{\Z}{\mathbb Z}

\newcommand\cL{\mathcal L}
\newcommand\cM{\mathcal M}

\def\P{\mathbb P}  
\def\q{\mathbb q}

\DeclareMathOperator\wgcd{wgcd}
\DeclareMathOperator\diam{diam}
\DeclareMathOperator\PGL{PGL}

\begin{document}
\title{Geometric Learning and Finsler Metrics in Weighted Projective Spaces}

\author{Tanush Shaska}
\address{Department of Computer Science and Engineering, \\ Oakland University, Rochester, MI, 48309}
\email{tanush@umich.edu}

\begin{abstract}
We introduce a hierarchical clustering framework for weighted projective spaces
$\P_{\q}$ built on Finsler geometry. From an optimization-based Finsler norm that
quotients out the weighted scaling action, we construct a scaling-invariant
distance $d_F([z], [w])$ and a rational analogue $d_{F,\Q}([z], [w])$ for points
of $\P_{\q}(\Q)$. The norm carries a shape parameter $p$: the case $p=2$ is
Riemannian and admits a closed-form distance, while $p\neq 2$ is genuinely
Finsler, and the metric and clustering guarantees below hold for every
$p\in[1,\infty)$. Whereas earlier work measured proximity in these spaces through
non-metric dissimilarities, we prove that $d_F$ satisfies the triangle inequality
and is therefore a genuine metric; this is what equips the induced clustering
with its theoretical guarantees, including monotone dendrograms and
Gromov--Hausdorff stability under perturbation of the data. The metric respects
the intrinsic scaling symmetry and weighted topology of $\P_{\q}$, avoiding the
distortions of a flat-space embedding. We develop the framework's arithmetic
applications---clustering rational points in the moduli space of genus two curves
and analyzing rational functions in arithmetic dynamics---and indicate
prospective extensions to quantum state spaces, where the weights $\q$ model
anisotropic noise. More broadly, the construction offers a rigorous metric
foundation for graded neural networks and related machine-learning techniques on
graded algebraic varieties.
\end{abstract} 


\maketitle



\section{Introduction}
Clustering data that lives on a non-Euclidean manifold becomes delicate when the
manifold carries an intricate scaling symmetry, as weighted projective spaces do.
These spaces---quotients of $\C^{n+1} \setminus \{0\}$ by a weighted scaling
action \cite{Dolgachev1982}---arise across arithmetic geometry, dynamical
systems, and data analysis, wherever projective symmetries organize the
underlying data. The moduli space of genus two curves, $\P_{(2,4,6,10)}$, is a
representative example: its coordinates are the Igusa invariants of degrees
$2, 4, 6, 10$. Clustering methods built on a Euclidean metric ignore this grading
and distort the intrinsic geometry, producing groupings that obscure the very
structure one hopes to detect.

This paper develops a clustering framework intrinsic to $\P_{\q}$. Its core is a
scaling-invariant Finsler norm (\cref{Finsler:norm}), a one-parameter family
indexed by a shape parameter $p$, whose geodesic integral defines a distance
$d_F([z], [w])$ (\cref{Finsler:distance}), together with a rational counterpart
$d_{F,\Q}([z], [w])$ for points of $\P_{\q}(\Q)$. The construction follows the
principles of Finsler geometry \cite{BaoChernShen2000} and, by minimizing over
the scaling direction of the quotient tangent space, yields a proximity measure
invariant under the weighted action. Our central result is that $d_F$ is a genuine
metric: an analysis of the quotient tangent space shows it satisfies the triangle
inequality. This is what separates the present approach from the earlier
non-metric dissimilarities of \cite{2024-3}--- the metric property is precisely
what lets the induced hierarchical clustering enjoy monotone dendrograms and
Gromov--Hausdorff stability, in place of the distortions of a flat-space
approximation.

This work belongs to a broader program developing machine learning for graded
spaces \cite{2024-2, 2025-5, 2025-6}. There, neural networks act on graded vector spaces
whose coordinates carry grades analogous to the weights
$\q = (q_0, \dots, q_n)$ of $\P_{\q}$, and graded neural networks weight features
by those grades. The metric developed here gives that program a geometrically
faithful notion of distance on $\P_{\q}$; conversely, such networks could learn
graded representations that feed the clustering algorithm.

The applications we pursue are arithmetic. In $\P_{(2,4,6,10)}(\Q)$ the algorithm
groups rational points by their Igusa invariants, recovering families such as
curves with $(n, n)$-split Jacobians and informing isogeny-based cryptography. In
arithmetic dynamics it clusters rational functions on $\P^1$, represented in a
weighted projective space, by their dynamical invariants \cite{2024-4}. The same
machinery extends in principle to other settings---including weighted models of
physical state spaces---where a grading encodes genuine structure.
The algorithm operates directly in $\P_{\q}$, building a dendrogram by merging
clusters under single- or average-linkage in $d_F$ \cite{Hastie2009}; its
correctness and stability rest on the metric axioms established here.
In the distinguished Riemannian case $p=2$, distances are available in closed
form---a projection in logarithmic coordinates followed by a low-dimensional
lattice search; for other $p$ the variational and discrete-path schemes apply.

\section{Preliminaries}

Weighted projective spaces serve as a foundational framework for our exploration of clustering algorithms in non-Euclidean manifolds, bridging arithmetic geometry, dynamical systems, and machine learning. These spaces, arising naturally in moduli problems where coordinates carry varying degrees, necessitate tailored distance measures that respect their quotient structure under weighted scalings. By establishing key concepts such as weights, heights, and dissimilarity measures, this section lays the groundwork for the Finsler metric introduced in subsequent sections, enabling a robust approach to geometric and arithmetic clustering that aligns with our broader program of developing graded neural networks for such spaces.

\subsection{Weighted projective spaces (WPS)}

Let \(\F\) be a field and \(q_0, q_1, \dots, q_n\) be positive integers called \emph{weights}. The tuple of weights is denoted by \(\q := (q_0, q_1, \dots, q_n)\). The \emph{weighted projective space} \(\P_{\q}\) is defined as the quotient space of \(\F^{n+1} \setminus \{0\}\) under the equivalence relation
\begin{equation} \label{eq:weighted-equivalence}
(z_0, z_1, \dots, z_n) \sim (\lambda^{q_0} z_0, \lambda^{q_1} z_1, \dots, \lambda^{q_n} z_n)
\end{equation}
for all \(\lambda \in \F^*\), where \(\F^*\) represents the multiplicative group of non-zero elements in \(\F\). A point in \(\P_{\q}\) is an equivalence class \([z] = [z_0 : z_1 : \dots : z_n]\), with the weights \(q_i\) governing the scaling of each coordinate. This construction extends the standard projective space, recovered when \(q_0 = q_1 = \dots = q_n = 1\). For the purposes of this paper, we primarily consider \(\F = \C\) for geometric clustering applications and \(\F = \Q\) for arithmetic geometry contexts, addressing both the geometric and Diophantine aspects of weighted projective spaces. The weights \(\q\) define a grading structure analogous to the graded vector spaces in \cite{2024-2, 2025-5}, positioning \(\P_{\q}\) as a natural framework for advancing machine learning techniques within our program for graded spaces.

\begin{rem} \label{rem-1}
The quotient structure of \(\P_{\q}\) under weighted scaling informs the construction of the Finsler metric in Section 3, ensuring that distances respect the manifold’s weighted geometry.
\end{rem}

\subsection{Heights on WPS}
\label{sec:heights}

To study the arithmetic properties of points in weighted projective spaces, we focus on rational points in \(\P_{\q}(\Q)\), where coordinates \(z_i \in \Q\) and the equivalence relation employs scalings \(\lambda \in \Q^*\). Drawing on the framework established in \cite{2022-1}, we normalize representatives of rational points to define a weighted height function that quantifies their arithmetic complexity. This graded structure supports our program’s goal, as outlined in \cite{2024-2}, to develop machine learning methods for arithmetic data analysis, potentially leveraging graded neural networks from \cite{2025-5}. For a point \([z] = [z_0 : z_1 : \dots : z_n] \in \P_{\q}(\Q)\), a normalized representative is chosen as \((x_0, x_1, \dots, x_n) \in \Z^{n+1} \setminus \{0\}\) such that the weighted greatest common divisor, denoted \(\wgcd(x_0, x_1, \dots, x_n)\), equals 1. The \(\wgcd\) is the largest positive integer \(d\) for which there exists a \(\lambda \in \Q^*\) satisfying \(\lambda^{q_i} x_i / d \in \Z\) for all \(i = 0, 1, \dots, n\). This normalization ensures that the representative is unique up to scaling by roots of unity in \(\Q^*\), providing a canonical form for arithmetic analysis.

The \emph{weighted height} of a point \([z] \in \P_{\q}(\Q)\), using its normalized representative \((x_0, x_1, \dots, x_n)\), is defined as
\begin{equation} \label{eq:weighted-height}
h_w([z]) = \max_{i=0,\dots,n} \left( \abs{x_i}^{1/q_i} \right).
\end{equation}
This height function is invariant under the weighted scaling action, since scaling the representative \((x_0, \dots, x_n)\) by \(\lambda \in \Q^*\) transforms each coordinate \(x_i\) to \(\lambda^{q_i} x_i\), and the term \(\abs{\lambda^{q_i} x_i}^{1/q_i} = \abs{\lambda} \abs{x_i}^{1/q_i}\) preserves the maximum up to a constant factor that cancels in the equivalence class. The weighted height serves as a measure of arithmetic complexity, enabling the ordering of rational points in databases or the analysis of Diophantine properties. For instance, in the moduli space \(\P_{(2,4,6,10)}(\Q)\), the weighted height quantifies the complexity of Igusa invariants, facilitating applications in arithmetic geometry as explored in \cite{2024-3}.

\begin{rem} \label{rem:arithmetic-normalization}
For rational points in \(\P_{\q}(\Q)\), the normalization using \(\wgcd = 1\) is crucial for arithmetic applications, such as clustering with the rational Finsler distance \(d_{F,\Q}([z], [w])\) in Section 5, distinct from the geometric normalization \(\sum_{k=0}^n q_k |z_k|^2 = 1\) used in Section 5.2.
\end{rem}

 \subsection{Dissimilarity Measures on Weighted Projective Spaces}

For clustering applications, we assume \(\F = \C\). To define a dissimilarity measure between points in \(\P_{\q}\) that respects the quotient structure, we first normalize representatives using a weighted norm.

Define the weighted norm \(N: \C^{n+1} \to [0, \infty)\) by
\begin{equation} \label{eq:weighted-norm}
N(z) = \sum_{k=0}^n q_k |z_k|^2.
\end{equation}

\begin{lem} \label{lem-1}
For \(z \in \C^{n+1} \setminus \{0\}\), there exists a unique \(a > 0\) such that \(N(a \cdot z) = 1\), where \(a \cdot z = (a^{q_0} z_0, \dots, a^{q_n} z_n)\).
\end{lem}

\begin{proof}
Consider the function \(g(a) = N(a \cdot z) = \sum_{k=0}^n q_k a^{2 q_k} |z_k|^2\) for \(a > 0\). Since \(z \neq 0\), there exists some \(k\) with \(z_k \neq 0\), so \(g(a) > 0\) for \(a > 0\). As \(a \to 0^+\), \(g(a) \to 0\) because each term \(a^{2 q_k} \to 0\) (as \(2 q_k \geq 2 > 0\)). As \(a \to \infty\), \(g(a) \to \infty\) since the highest-degree term dominates. The derivative \(g'(a) = \sum_{k=0}^n 2 q_k^2 a^{2 q_k - 1} |z_k|^2 > 0\) for \(a > 0\), so \(g\) is strictly increasing. By the intermediate value theorem, there exists a unique \(a > 0\) with \(g(a) = 1\).
\end{proof}

Let \(\tilde{z} = a \cdot z\) be the normalized representative with \(N(\tilde{z}) = 1\). This \(\tilde{z}\) is unique up to multiplication by a phase factor \(e^{i \theta}\) for \(\theta \in [0, 2\pi)\), since scaling by \(e^{i \theta}\) preserves the weighted norm (as \(|e^{i \theta}| = 1\)). To fix uniqueness, adjust the phase so that the first non-zero coordinate is real and positive.

Similarly define \(\tilde{w}\) for \(w\).

The \emph{dissimilarity measure} between \([z], [w] \in \P_{\q}\) is
\begin{equation} \label{eq:dissimilarity-measure}
d([z], [w]) = \min_{|\phi| = 1} \norm{ \tilde{z} - \phi \cdot \tilde{w} },
\end{equation}
where \(\norm{\cdot}\) is the Euclidean norm in \(\C^{n+1}\), i.e., \(\norm{v} = \left( \sum_{k=0}^n |v_k|^2 \right)^{1/2}\), and \(\phi \cdot \tilde{w} = (\phi^{q_0} \tilde{w}_0, \dots, \phi^{q_n} \tilde{w}_n)\).

This quantifies the minimal Euclidean separation between normalized representatives under phase adjustments, respecting the quotient structure of \(\P_{\q}\).

\begin{lem} \label{lem-2}
The minimum in the definition of \(d([z], [w])\) is attained.
\end{lem}

\begin{proof}
The set \(\{ \phi \in \C : |\phi| = 1 \}\) is the unit circle \(S^1\), which is compact in the subspace topology of \(\C\). The function \(h(\phi) = \norm{ \tilde{z} - \phi \cdot \tilde{w} }\) is continuous on \(S^1\) because the Euclidean norm is continuous, and the map \(\phi \mapsto \phi \cdot \tilde{w}\) is continuous (as it is polynomial in \(\phi\)). By the extreme value theorem, \(h\) attains its minimum on the compact set \(S^1\).
\end{proof}

We next establish that this measure is well-defined and finite.

\begin{lem} \label{lem-3}
For any \([z], [w] \in \P_{\q}\), the dissimilarity measure \(d([z], [w])\) is well-defined (independent of representatives and phase conventions) and finite.
\end{lem}

\begin{proof}
Let \(z' = \nu \cdot z\) for \(\nu \in \C^*\). The normalizing scalar \(a'\) for \(z'\) satisfies the same equation as for \(z\) but scaled by \(1/|\nu|\), since \(N(\nu \cdot z) = \sum q_k |\nu|^{2 q_k} |z_k|^2\). Thus, \(\tilde{z}' = e^{i\theta} \tilde{z}\) for \(\theta = \arg(\nu)\), up to the phase convention which ensures the first non-zero coordinate is real and positive, absorbing \(\theta\). Similarly for \(w'\). Then,
\[
\min_{|\phi|=1} \norm{ \tilde{z}' - \phi \cdot \tilde{w}' } = \min_{|\phi|=1} \norm{ e^{i\theta} \tilde{z} - \phi e^{i\psi} \cdot \tilde{w} } = \min_{|\phi|=1} \norm{ \tilde{z} - e^{-i\theta} \phi e^{i\psi} \cdot \tilde{w} }.
\]
The map \(\phi \mapsto e^{-i\theta} \phi e^{i\psi}\) is a homeomorphism of \(S^1\) onto itself (rotation and inversion preserve the circle), so it preserves minima of continuous functions. Thus, the minimum is unchanged, and \(d\) is independent of representatives and phase conventions.

For finiteness: Since \(N(\tilde{z}) = N(\tilde{w}) = 1\), we bound the Euclidean norm. Note that \(\norm{\tilde{z}}^2 = \sum_{k=0}^n |\tilde{z}_k|^2 \leq \left( \max_{0 \leq k \leq n} \frac{1}{q_k} \right) \sum_{k=0}^n q_k |\tilde{z}_k|^2 = \left( \max_{0 \leq k \leq n} \frac{1}{q_k} \right) \cdot 1 < \infty\), since \(q_k \geq 1\) are finite positive integers. Similarly for \(\tilde{w}\). Thus, \(d([z], [w]) \leq \norm{\tilde{z}} + \norm{\tilde{w}} < \infty\).
\end{proof}

Finally, we prove the key properties of the dissimilarity measure.

\begin{lem} \label{lem-4}
The dissimilarity measure \(d\) satisfies:
\begin{enumerate}
\item Non-negativity: \(d([z], [w]) \geq 0\),
\item Symmetry: \(d([z], [w]) = d([w], [z])\),
\item Separation: \(d([z], [w]) = 0\) if and only if \([z] = [w]\).
\end{enumerate}
\end{lem}

\begin{proof}
Non-negativity follows directly from the definition, as \(d([z], [w])\) is the minimum of non-negative Euclidean norms.
 
For symmetry,
\[
d([z], [w]) = \min_{|\phi|=1} \norm{ \tilde{z} - \phi \cdot \tilde{w} } = \min_{|\phi|=1} \norm{ \phi^{-1} \cdot \tilde{z} - \tilde{w} },
\]
since multiplication by \(\phi\) (with \(|\phi| = 1\)) is an isometry for the Euclidean norm: \(\norm{ \phi \cdot v }^2 = \sum_{k=0}^n |\phi^{q_k} v_k|^2 = \sum_{k=0}^n |\phi|^{2 q_k} |v_k|^2 = \sum_{k=0}^n |v_k|^2 = \norm{v}^2\), as \(|\phi| = 1\). As \(\phi\) ranges over \(S^1\), so does \(\phi^{-1} = \overline{\phi}\), yielding \(d([w], [z])\).

 For separation: If \([z] = [w]\), there exists \(\nu \in \C^*\) such that \(z = \nu \cdot w\). Normalization preserves this up to phase: the equations for the normalizing scalars coincide after accounting for \(|\nu|\), so \(\tilde{z} = e^{i \theta} \cdot \tilde{w}\) for some \(\theta\). Thus, \(d([z], [w]) \leq \norm{ \tilde{z} - e^{i \theta} \cdot \tilde{w} } = 0\), and since \(d \geq 0\), equality holds.

Conversely, if \(d([z], [w]) = 0\), there exists \(|\phi| = 1\) such that \(\norm{ \tilde{z} - \phi \cdot \tilde{w} } = 0\), so \(\tilde{z} = \phi \cdot \tilde{w}\). Reversing normalization, the scalars and phase imply \(z = \nu \cdot w\) for some \(\nu \in \C^*\), hence \([z] = [w]\).

\end{proof}

This dissimilarity measure is particularly suitable for clustering in weighted projective spaces because it respects the equivalence relation defined by the weights. Specifically, it is invariant under the weighted scaling action:

\begin{lem} \label{lem-5}
For any \(\lambda, \mu \in \C^*\),
\begin{equation} \label{eq:dissimilarity-invariance}
d([z], [w]) = d([\lambda^{q_0} z_0, \dots, \lambda^{q_n} z_n], [\mu^{q_0} w_0, \dots, \mu^{q_n} w_n]).
\end{equation}
\end{lem}

\begin{proof}
The right-hand side is \(d([\lambda \cdot z], [\mu \cdot w])\). By the well-definedness lemma above, \(d\) is independent of the choice of representatives, so replacing \(z\) by \(\lambda \cdot z\) and \(w\) by \(\mu \cdot w\) does not change the value of \(d\).
\end{proof}

This ensures that the clustering is based on the intrinsic geometry of the space, rather than on specific choices of representatives for the points. In many applications, such as image analysis or genomic data, the data points naturally reside in a weighted projective space due to inherent symmetries or scaling properties. By employing a dissimilarity measure that accounts for these properties, our clustering algorithm can effectively group points that are similar in a geometrically meaningful way. This dissimilarity measure serves as a valid tool for clustering purposes, enabling algorithms such as hierarchical clustering to partition the data effectively.

\begin{rem} \label{rem:computing-dissimilarity}
Computing the exact value of \( d([z], [w]) \) involves solving an optimization problem over the unit circle, which can be computationally intensive. In practice, we approximate this minimum by sampling a finite set of \(\phi\) values or by employing numerical optimization methods to find sufficiently close approximations.
\end{rem}

Working directly in the weighted projective space allows us to leverage the inherent geometric structure of the data, which can lead to more efficient and accurate clustering compared to traditional methods that might require projecting the data into a different space. By preserving the weighted scaling equivalences, our approach can capture symmetries and invariances that are crucial in applications such as computer vision and genomic data analysis. Furthermore, as demonstrated in \cite{2024-3} and \cite{2024-4}, this direct approach can offer computational advantages, particularly in high-dimensional or heterogeneous data settings.


\subsection{Rational Points in Weighted Projective Spaces}

For arithmetic applications, we consider the subset \(\P_{\q}(\Q)\) of points with rational coordinates \( z_i \in \Q \), where the equivalence relation uses scalings \(\lambda \in \Q^*\). Points in \(\P_{\q}(\Q)\) are normalized using the weighted greatest common divisor to facilitate arithmetic analysis. For a point \([z] = [z_0 : z_1 : \dots : z_n] \in \P_{\q}(\Q)\), we select a representative \((x_0, x_1, \dots, x_n) \in \Z^{n+1} \setminus \{0\}\) such that the weighted greatest common divisor \(\wgcd(x_0, x_1, \dots, x_n) = 1\). The \(\wgcd\) is defined as the largest positive integer \( d \) for which there exists a \(\lambda \in \Q^*\) satisfying \(\lambda^{q_i} x_i / d \in \Z\) for all \( i = 0, 1, \dots, n \). This normalization ensures a canonical representative, unique up to scaling by units in \(\Q^*\) (i.e., \(\pm 1\)).

The \emph{weighted height} of a point \([z] \in \P_{\q}(\Q)\), using its normalized representative \((x_0, x_1, \dots, x_n)\), is defined as
\begin{equation} \label{eq:weighted-height-rational}
h_w([z]) = \max_{i=0,\dots,n} \left( \abs{x_i}^{1/q_i} \right).
\end{equation}
This height is invariant under the weighted scaling action, as scaling \((x_0, \dots, x_n)\) by \(\lambda \in \Q^*\) yields coordinates \(\lambda^{q_i} x_i\), and \(\abs{\lambda^{q_i} x_i}^{1/q_i} = \abs{\lambda} \abs{x_i}^{1/q_i}\), preserving the maximum up to a factor that cancels in the equivalence class. The weighted height quantifies the arithmetic complexity of rational points, enabling their ordering in databases or the study of Diophantine properties, such as in moduli spaces of algebraic curves as explored in \cite{2024-3}.

For rational points, we define a \emph{rational dissimilarity measure} between \([z], [w] \in \P_{\q}(\Q)\) as
\begin{equation} \label{eq:rational-dissimilarity}
d_\Q([z], [w]) = \min_{\phi \in \{1, -1\}} \left( \sum_{i=0}^n \abs{x_i - \phi^{q_i} y_i}^2 \right)^{1/2},
\end{equation}
where \((x_0, x_1, \dots, x_n)\) and \((y_0, y_1, \dots, y_n)\) are normalized representatives with \( x_i, y_i \in \Z \) and \(\wgcd(x_0, x_1, \dots, x_n) = \wgcd(y_0, y_1, \dots, y_n) = 1\), and \(\phi^{q_i} y_i\) denotes the weighted scaling by \(\phi\). This measure extends the geometric clustering framework to rational points, respecting the weighted scaling action over \(\Q^*\).

\begin{lem} \label{lem-6}
The function \( d_\Q([z], [w]) \) on \(\P_{\q}(\Q)\), defined as
\begin{equation} \label{eq:rational-dissimilarity-definition}
d_\Q([z], [w]) = \min_{\phi \in \{1, -1\}} \left( \sum_{i=0}^n \abs{x_i - \phi^{q_i} y_i}^2 \right)^{1/2},
\end{equation}
where \((x_0, x_1, \dots, x_n)\) and \((y_0, y_1, \dots, y_n)\) are normalized representatives with \\
\(\wgcd(x_0, x_1, \dots, x_n) = \wgcd(y_0, y_1, \dots, y_n) = 1\), satisfies:
\begin{enumerate}
    \item \( d_\Q([z], [w]) \geq 0 \),
    \item \( d_\Q([z], [w]) = d_\Q([w], [z]) \),
    \item \( d_\Q([z], [w]) = 0 \) if and only if \( [z] = [w] \).
\end{enumerate}
\end{lem}

\begin{proof}
The non-negativity follows directly, as \( d_\Q([z], [w]) \) is the minimum of non-negative Euclidean norms.

For symmetry,
\[
d_\Q([z], [w]) = \min_{\phi \in \{1, -1\}} \left( \sum_{i=0}^n \abs{x_i - \phi^{q_i} y_i}^2 \right)^{1/2} = \min_{\phi \in \{1, -1\}} \left( \sum_{i=0}^n \abs{\phi^{q_i} y_i - x_i}^2 \right)^{1/2},
\]
since \(\abs{a - b}^2 = \abs{b - a}^2\). As \(\phi\) ranges over \(\{1, -1\}\), so does \(-\phi\), and \((-\phi)^{q_i} = (-1)^{q_i} \phi^{q_i}\), which equals \(\phi^{q_i}\) if \(q_i\) is even and \(-\phi^{q_i}\) if odd. However, since the minimum is over both signs, the values coincide, yielding \( d_\Q([w], [z]) = d_\Q([z], [w]) \).

For separation: If \([z] = [w]\), there exists \(\alpha \in \Q^*\) such that \( y_i = \alpha^{q_i} x_i \) for all \( i \). Since representatives are normalized with \(\wgcd = 1\), \(\alpha = \pm 1\) (the units in \(\Q^*\)). Choose \(\phi = \alpha\), giving
\[
x_i - \phi^{q_i} y_i = x_i - \alpha^{q_i} (\alpha^{q_i} x_i) = x_i - x_i = 0.
\]
Thus, the norm is 0 for this \(\phi\), so \( d_\Q([z], [w]) = 0 \). Conversely, if \( d_\Q([z], [w]) = 0 \), there exists \(\phi \in \{1, -1\}\) such that \(\sum_{i=0}^n \abs{x_i - \phi^{q_i} y_i}^2 = 0\), implying \( x_i = \phi^{q_i} y_i \) for all \( i \). Thus, \( x = \phi \cdot y \), and since \(\phi \in \Q^*\), \([z] = [w]\).
\end{proof}
 
\begin{rem} \label{rem-3}
The dissimilarity measures \( d([z], [w]) \) and \( d_\Q([z], [w]) \),  are effective for clustering, but they do not satisfy the triangle inequality.
The minimization over unit circle phases \(\phi\) for \( d([z], [w]) \) and over rational units \(\phi \in \{1, -1\}\) for \( d_\Q([z], [w]) \) is optimized independently for each pair of points. This independent optimization can lead to configurations where the triangle inequality 
does not hold, as the phases that minimize dissimilarities for different pairs may not align additively.

The weighted height and dissimilarity measures provide a dual perspective: the height \( h_w([z]) \) orders points by arithmetic complexity, while the dissimilarity measures \( d([z], [w]) \) and \( d_\Q([z], [w]) \) group points geometrically, enhancing data analysis in contexts like moduli spaces where both geometric and arithmetic structures are significant, as demonstrated in \cite{2024-3}.
\end{rem}

\section{Finsler Metric on Weighted Projective Spaces}

To define a distance on the weighted projective space $\P_q$, we introduce a Finsler metric that induces a true metric, offering a theoretical framework for potential clustering applications. The weighted projective space $\P_q$, as a quotient of $\C^{n+1} \setminus \{0\}$ under a weighted scaling action, requires careful consideration of curves and their tangent vectors to define a Finsler metric. We begin by detailing the process of lifting curves from $\P_q$ to $\C^{n+1} \setminus \{0\}$, which leads to the definition of the tangent vector $\dot{\gamma}(t)$, essential for the Finsler distance.

\subsection{Curves and Lifting in Weighted Projective Spaces}

The weighted projective space $\P_q$ is defined as the quotient of $\C^{n+1} \setminus \{0\}$ under the equivalence relation $(z_0, z_1, \dots , z_n) \sim (\lambda^{q_0} z_0, \lambda^{q_1} z_1, \dots , \lambda^{q_n} z_n)$ for $\lambda \in \C^*$, where $q_0, q_1, \dots , q_n$ are positive integers called weights, denoted by $q = (q_0, q_1, \dots , q_n)$. A point $[z] \in \P_q$ is an equivalence class $[z_0 : z_1 : \cdots : z_n]$, represented by a vector $z = (z_0, z_1, \dots , z_n) \in \C^{n+1} \setminus \{0\}$.

To define a Finsler metric, we consider smooth curves $\gamma : [0, 1] \to \P_q$, which connect points $[z], [w] \in \P_q$ and whose tangent vectors are used to measure distances in the Finsler geometry framework, as described in \cite{Shen2012}.

A curve $\gamma : [0, 1] \to \P_q$ is smooth if, in local coordinates on $\P_q$, its component functions are smooth (i.e., infinitely differentiable). Since $\P_q$ is a complex manifold (or orbifold for non-coprime weights), smoothness implies that $\gamma(t)$ varies continuously and differentiably in the quotient space. However, $\P_q$ is defined as a quotient, so to work with $\gamma(t)$, we must lift it to a curve in the covering space $\C^{n+1} \setminus \{0\}$, where differentiation is straightforward. The lifting process constructs a representative curve whose derivative defines the tangent vector $\dot{\gamma}(t)$.

Given a smooth curve $\gamma : [0, 1] \to \P_q$ with $\gamma(0) = [z]$ and $\gamma(1) = [w]$, a lift of $\gamma(t)$ is a smooth curve
\[
z(t) = (z_0(t), z_1(t), \dots , z_n(t)) \in \C^{n+1} \setminus \{0\}
\]
such that $[z(t)] = \gamma(t)$ for all $t \in [0, 1]$. That is, $z(t) \neq 0$ and maps to $\gamma(t)$ under the quotient map $\pi : \C^{n+1} \setminus \{0\} \to \P_q$, defined by $\pi(z) = [z]$. The lift is not unique, as any scaled curve
\[
z'(t) = (\lambda(t)^{q_0} z_0(t), \lambda(t)^{q_1} z_1(t), \dots , \lambda(t)^{q_n} z_n(t)),
\]
where $\lambda(t) \in \C^*$ is a smooth function, also satisfies $[z'(t)] = \gamma(t)$.

To construct a lift, consider a local coordinate chart on $\P_q$. For a point $[z] \in \P_q$, suppose $z_k \neq 0$ for some $k$. In the chart $U_k = \{[z_0 : \cdots : z_n] \in \P_q \mid z_k \neq 0\}$, we can represent $[z]$ by normalizing the $k$-th coordinate to 1, yielding coordinates
\[
\left( \frac{z_0}{z_k^{q_0 / q_k}}, \dots , \frac{z_{k-1}}{z_k^{q_{k-1} / q_k}}, 1, \frac{z_{k+1}}{z_k^{q_{k+1} / q_k}}, \dots , \frac{z_n}{z_k^{q_n / q_k}} \right).
\]
If $\gamma(t)$ lies in $U_k$, we can choose a representative $z(t) = (z_0(t), \dots , z_n(t))$ with $z_k(t) = 1$, and smoothness of $\gamma(t)$ ensures the other coordinates $z_i(t) / z_k(t)^{q_i / q_k}$ are smooth functions of $t$. For a general curve $\gamma(t)$, which may exit one chart, we cover $[0, 1]$ with finitely many intervals where $\gamma(t)$ lies in charts $U_{k_i}$, and construct $z(t)$ piecewise, ensuring smoothness by adjusting scalings $\lambda(t) \in \C^*$ to glue the pieces across chart transitions. Since $\P_q$ is a smooth manifold (or orbifold), such a smooth lift exists, as the quotient map $\pi$ is a submersion \cite{Dolgachev1982}.

The tangent vector $\dot{\gamma}(t)$ is defined via the lift $z(t)$. The derivative of the lifted curve is
\[
\dot{z}(t) = \left( \frac{d}{dt} z_0(t), \frac{d}{dt} z_1(t), \dots , \frac{d}{dt} z_n(t) \right) \in \C^{n+1},
\]
where each $\dot{z}_k(t) = \frac{d}{dt} z_k(t) \in \C$ is the derivative of the coordinate function $z_k(t)$. This vector $\dot{z}(t)$ lies in the tangent space
\[
T_{z(t)}(\C^{n+1} \setminus \{0\}) \simeq \C^{n+1}.
\]
In the quotient space $\P_q$, the tangent space $T_{[\gamma(t)]} \P_q$ at $[\gamma(t)] = [z(t)]$ is the quotient of $T_{z(t)}(\C^{n+1}\setminus \{0\})$ by the tangent vectors of the scaling action’s orbits. The Finsler norm $F([z], v)$, defined below, is invariant under this action, so we compute
\[
F(\gamma(t), \dot{\gamma}(t)) = F([z(t)], \dot{z}(t)),
\]
where $\dot{\gamma}(t)$ is represented by $\dot{z}(t)$ in the quotient tangent space.

To formalize $\dot{\gamma}(t)$, consider the differential of the quotient map
\[
\pi : \C^{n+1} \setminus \{0\} \to \P_q.
\]
For a point $z(t) \in \C^{n+1} \setminus \{0\}$, the tangent vector $\dot{z}(t)$ is mapped to $\dot{\gamma}(t) \in T_{[\gamma(t)]} \P_q$ via
\[
d\pi_{z(t)} : T_{z(t)}(\C^{n+1} \setminus \{0\}) \to T_{[\gamma(t)]} \P_q.
\]
The scaling action $z \mapsto (\lambda^{q_0} z_0, \dots , \lambda^{q_n} z_n)$ generates an orbit through $z(t)$, and vectors tangent to this orbit, such as $(\alpha q_0 z_0(t), \dots , \alpha q_n z_n(t))$, are quotiented out. If we choose a different lift
\[
z'(t) = (\lambda(t)^{q_0} z_0(t), \dots , \lambda(t)^{q_n} z_n(t)),
\]
the derivative is
\[
\dot{z}'(t) = \left( q_0 \lambda(t)^{q_0-1} \dot{\lambda}(t) z_0(t) + \lambda(t)^{q_0} \dot{z}_0(t), \dots , q_n \lambda(t)^{q_n-1} \dot{\lambda}(t) z_n(t) + \lambda(t)^{q_n} \dot{z}_n(t) \right).
\]
The Finsler norm $F([z], v)$ is designed to be invariant under scaling, ensuring
\[
F([z'(t)], \dot{z}'(t)) = F([z(t)], \dot{z}(t)),
\]
so $\dot{\gamma}(t)$ is well-defined as the equivalence class of $\dot{z}(t)$ in $T_{[\gamma(t)]} \P_q$. This lifting process, rooted in the quotient structure of $\P_q$ as described in \cite{Dolgachev1982}, allows us to define the Finsler distance using tangent vectors derived from lifted curves, as detailed in \cite{Shen2012}.

\subsection{The Finsler Norm and Metric}

Write
\[
  \P_q^{\circ} := \{\, [z] \in \P_q : z_k \neq 0 \text{ for all } k \,\}
\]
for the open dense subset of $\P_q$ on which every coordinate is nonzero---the
orbit of the big torus, equivalently the complement of the coordinate
hyperplanes $\{z_k = 0\}$. For a point $[z] \in \P_q^{\circ}$ with representative
$z = (z_0, z_1, \dots , z_n) \in (\C^{*})^{n+1}$, and a tangent vector
$v = (v_0, v_1, \dots , v_n) \in \C^{n+1}$, define the Finsler norm

\begin{equation}
\label{Finsler:norm}
F_p([z], v) = \min_{\alpha \in \C}
 \left( \sum_{k=0}^n \left| \frac{v_k - \alpha q_k z_k}{z_k} \right|^p \right)^{1/p},
 \qquad p \in [1,\infty).
\end{equation}
Here \(p\in[1,\infty)\) is a fixed shape parameter; we write \(F := F_p\) and
suppress \(p\) when no confusion arises. The value \(p=2\) recovers the Euclidean
objective and is the Riemannian special case (\cref{rem-4}); for
\(p\neq 2\) the metric is genuinely Finsler.

The norm divides coordinatewise by $z_k$, so $\P_q^{\circ}$ is its natural
domain. In the logarithmic coordinates $\zeta_k = \log z_k$ it is a
translation-invariant norm with the scaling direction $\alpha\,(q_0,\dots,q_n)$
projected out; a path approaching a coordinate hyperplane has infinite length,
since the uniform scaling shift cannot offset a divergence in only some of the
coordinates. The hyperplanes therefore lie at infinite $d_F$-distance. Each is
itself a weighted projective space of lower dimension, carrying a Finsler norm of
the same form, so the construction applies to the boundary strata
recursively.\footnote{For non-coprime weights $\P_q$ also carries cyclic
quotient singularities; on the singular locus the Finsler structure is
understood in the orbifold ($V$-manifold) sense and extended continuously from
the smooth part, following the framework of weighted manifolds in
\cite{santos2014}. This is independent of the coordinate-hyperplane restriction
above.}
This norm, weighted by the grading $q$, aligns with the graded vector spaces in
\cite{2024-2,2025-5}, enabling potential integration with graded neural networks
for clustering in $\P_q$.

The induced Finsler distance between points $[z], [w] \in \P_q^{\circ}$ is defined as
\begin{equation}
\label{Finsler:distance}
d_F ([z], [w]) = \inf_{\gamma} \int_0^1 F(\gamma(t), \dot{\gamma}(t)) \, dt,
\end{equation}
where $\gamma : [0, 1] \to \P_q^{\circ}$ is a smooth curve satisfying
$\gamma(0) = [z]$ and $\gamma(1) = [w]$. This construction, inspired by Finsler
geometry principles as described in \cite{Shen2012}, adapts the weighted
structure of $\P_q$ as a quotient space, as studied in \cite{Dolgachev1982}, to
provide a true metric, enhancing geometric analysis in clustering contexts.

\begin{lem}
\label{lem-7}
The function $F([z], v)$ defines a Finsler norm on $\P_q^{\circ}$ for every
$p \in [1,\infty)$, and the induced distance $d_F ([z], [w])$ is a well-defined,
finite metric on $\P_q^{\circ}$ satisfying:
\begin{enumerate}
\item $d_F ([z], [w]) \geq 0$,
\item $d_F ([z], [w]) = d_F ([w], [z])$,
\item $d_F ([z], [w]) = 0$ if and only if $[z] = [w]$,
\item $d_F ([z], [v]) \leq d_F ([z], [w]) + d_F ([w], [v])$.
\end{enumerate}
\end{lem}

\begin{proof}
We first show \(F_p([z],\cdot)\) is a norm on the quotient tangent space. Fix
\([z]\in\P_q^{\circ}\) and set \(w = w(v) = (v_0/z_0,\dots,v_n/z_n)\); since each
\(z_k\neq 0\), the map \(v\mapsto w\) is a \(\C\)-linear isomorphism of \(\C^{n+1}\).
Let \(\norm{\cdot}_p\) be the \(\ell^p\) norm on \(\C^{n+1}\) and \(L=\C\cdot q\) the
line spanned by \(q=(q_0,\dots,q_n)\). Then \cref{Finsler:norm} reads
\[
 F_p([z],v) \;=\; \min_{\alpha\in\C}\norm{w-\alpha q}_p
 \;=\; \operatorname{dist}_{\norm{\cdot}_p}(w,L),
\]
the quotient norm of \((\C^{n+1},\norm{\cdot}_p)\) by \(L\). A quotient of a norm by
a subspace is again a norm, so \(F_p\) is positively homogeneous,
\(F_p([z],\lambda v)=\abs{\lambda}F_p([z],v)\) for \(\lambda\in\C\); subadditive,
\(F_p([z],v+v')\le F_p([z],v)+F_p([z],v')\); reversible,
\(F_p([z],-v)=F_p([z],v)\); and positive off \(L\). The orbit direction
\(v_k=\alpha q_k z_k\) maps under \(v\mapsto w\) precisely onto \(L\), so \(F_p\)
descends to a genuine norm on the quotient tangent space \(T_{[z]}\P_q^{\circ}\),
vanishing exactly on the quotiented orbit direction.

For invariance under the scaling action, let \(z'=(\lambda^{q_k}z_k)_k\) and
\(v'=(\lambda^{q_k}v_k)_k\) with \(\lambda\in\C^{*}\). Then
\(\,(v'_k-\alpha q_k z'_k)/z'_k = (v_k-\alpha q_k z_k)/z_k\,\) for every \(k\), so the
objective and its minimum over \(\alpha\) are unchanged: \(F_p([z'],v')=F_p([z],v)\).

To show that $d_F ([z], [w])$ is well-defined and finite, consider a smooth curve $\gamma : [0, 1] \to \P_q^{\circ}$ with $\gamma(0) = [z]$ and $\gamma(1) = [w]$. As described in the lifting process, we lift $\gamma(t)$ to a smooth curve $z(t) \in (\C^{*})^{n+1}$ such that $[z(t)] = \gamma(t)$, and the tangent vector $\dot{\gamma}(t)$ is represented by
\[
\dot{z}(t) = (\dot{z}_0(t), \dots , \dot{z}_n(t)).
\]
The integrand $F(\gamma(t), \dot{\gamma}(t)) = F([z(t)], \dot{z}(t))$ is continuous, as $z(t)$ and $\dot{z}(t)$ are smooth and $F$ is continuous on the tangent bundle of $\P_q^{\circ}$ minus the zero section. Since $[0, 1]$ is compact, the integral
\[
\int_0^1 F(\gamma(t), \dot{\gamma}(t)) \, dt
\]
exists and is finite. The space $\P_q^{\circ}$ is path-connected, being the quotient of the connected set $(\C^{*})^{n+1}$ under the scaling action, so such curves exist; the infimum over them is therefore finite and nonnegative. Invariance of $F([z], v)$ under the scaling action ensures the integral depends only on the equivalence classes $[z]$ and $[w]$, making $d_F ([z], [w])$ well-defined.

The metric properties of $d_F ([z], [w])$ are established as follows. For non-negativity, since $F([z], v) \geq 0$, the integral $\int_0^1 F(\gamma(t), \dot{\gamma}(t)) \, dt \geq 0$, so $d_F ([z], [w]) \geq 0$. For symmetry, consider a curve $\gamma(t)$ from $[z]$ to $[w]$. The reversed curve $\gamma(1-t)$ from $[w]$ to $[z]$ has tangent vector $-\dot{\gamma}(1-t)$. Since $F([z], -v) = F([z], v)$ due to the absolute value in the norm, we have
\[
F(\gamma(1 - t), -\dot{\gamma}(1 - t)) = F(\gamma(t), \dot{\gamma}(t)),
\]
so the integral along $\gamma(1-t)$ equals that along $\gamma(t)$. Thus, $d_F ([z], [w]) = d_F ([w], [z])$. If $[z] = [w]$, the constant path has zero length, so $d_F ([z], [w]) = 0$. Conversely, if $d_F ([z], [w]) = 0$ then $[z]$ and $[w]$ are joined by curves of arbitrarily small length; since $F$ restricts to a genuine (positive-definite) norm on the horizontal distribution, the induced length distance is positive between distinct points of $\P_q^{\circ}$, forcing $[z] = [w]$.

For the triangle inequality, fix any smooth curves $\gamma_1 : [0,1] \to \P_q^{\circ}$ from $[z]$ to $[w]$ and $\gamma_2 : [0,1] \to \P_q^{\circ}$ from $[w]$ to $[v]$. Their concatenation $\gamma : [0,2] \to \P_q^{\circ}$, with $\gamma(t) = \gamma_1(t)$ for $t \in [0,1]$ and $\gamma(t) = \gamma_2(t-1)$ for $t \in [1,2]$, satisfies
\[
\int_0^2 F(\gamma(t), \dot{\gamma}(t)) \, dt = \int_0^1 F(\gamma_1, \dot{\gamma}_1) \, dt + \int_0^1 F(\gamma_2, \dot{\gamma}_2) \, dt,
\]
and $d_F([z],[v]) \leq \int_0^2 F(\gamma, \dot\gamma)\,dt$ since the left-hand side is an infimum over all curves from $[z]$ to $[v]$. Taking the infimum over $\gamma_1$ and $\gamma_2$ gives $d_F ([z], [v]) \leq d_F ([z], [w]) + d_F ([w], [v])$. Hence $d_F$ is a well-defined, finite metric on $\P_q^{\circ}$.
\end{proof}

\begin{rem}
\label{rem-4}
For \(p=2\) the minimization over \(\alpha\) in \cref{Finsler:norm} is the orthogonal
projection of \(w=(v_k/z_k)_k\) off the line \(L=\C\cdot(q_0,\dots,q_n)\), giving the
closed form
\[
 F_2([z],v)^2 \;=\; \norm{w}_2^2 - \frac{\abs{\langle w,q\rangle}^2}{\norm{q}_2^2},
\]
a positive semidefinite Hermitian quadratic form in \(v\), degenerate exactly along
the orbit direction. Hence \(F_2\) is induced by a Hermitian inner product on the
horizontal distribution and \((\P_q^{\circ}, d_{F_2})\) is a Hermitian Riemannian
manifold; this is the case in which the closed form of \cref{prop:closed-form} is
available. For \(p\neq 2\) the objective is not quadratic in \(v\), so \(F_p\) is not
induced by any inner product and \((\P_q^{\circ}, d_{F_p})\) is a genuinely
non-Riemannian Finsler manifold. By \cref{lem-7} the metric axioms hold for every
\(p\in[1,\infty)\), and the clustering results of Section~5 depend only on those
axioms; the choice of \(p\) is therefore free, with \(p=2\) the computationally
distinguished Riemannian member of the family.
\end{rem}

\begin{rem}
\label{rem-5}
The metric properties in \cref{lem-7} require only that \(F_p\) be a norm, which holds
for all \(p\in[1,\infty)\). The smooth Finsler-geodesic theory invoked in
\cref{lem-9} additionally needs \(F_p\) to be \(C^2\) and strongly convex off the
orbit direction. This holds when \(p\) is an even integer (notably \(p=2,4\)), where
\(\abs{s}^p=(s\bar s)^{p/2}\) is a polynomial in the real coordinates; for other
\(p\in(1,\infty)\), \(F_p\) is a uniformly convex norm but fails to be \(C^2\) on the
loci where a coordinate of \(w\) vanishes. For such \(p\), \(d_{F_p}\) is still a
genuine metric and \cref{lem-9} is read as minimization over rectifiable paths, or
\(F_p\) is replaced by a smooth strongly convex norm in its bi-Lipschitz class
without affecting any clustering result.
\end{rem}

\begin{rem}
\label{rem-6}
The Finsler distance $d_F ([z], [w])$ provides a true metric on $\P_q^{\circ}$,
satisfying all metric axioms. The ambient $\P_q$ is compact, and the coordinate
hyperplanes lie at infinite $d_F$-distance; hence any sequence in a closed
$d_F$-ball stays bounded away from these hyperplanes, and any limit point it has
in $\P_q$ lies in $\P_q^{\circ}$. A closed $d_F$-ball is therefore a closed subset
of the compact $\P_q$ contained in $\P_q^{\circ}$, hence compact, so
$(\P_q^{\circ}, d_F)$ is proper and in particular complete; this is the
completeness used for the existence of minimizing geodesics in the sequel (for $p$
in the regular range of \cref{rem-5}). Computing $d_F ([z], [w])$ requires
numerical optimization to determine geodesics, achievable through variational
methods or discrete path approximations, as discussed in \cite{Shen2012}.
\end{rem}

\subsection{Finsler Metric on Rational Points}

For arithmetic applications we restrict to the rational points of the open locus,
\[
  \P_q^{\circ}(\Q) := \{\, [z] \in \P_q(\Q) : x_k \neq 0 \text{ for all } k
  \text{ in a normalized representative} \,\},
\]
where $x = (x_0, \dots, x_n) \in \mathbb{Z}^{n+1}$ is the representative with
$\wgcd(x_0, \dots, x_n) = 1$. For $[z] \in \P_q^{\circ}(\Q)$ and a rational tangent
vector $v = (v_0, \dots, v_n) \in \Q^{n+1}$, define the \textbf{rational Finsler
norm}
\begin{equation}
\label{rational:norm}
F_{p,\Q}([z], v) = \min_{\alpha \in \Q}
 \left( \sum_{k=0}^n \left| \frac{v_k - \alpha q_k x_k}{x_k} \right|^p \right)^{1/p}.
\end{equation}
As in the complex case we write $F_{\Q} := F_{p,\Q}$ and suppress $p$. Since
$\Q \subset \C$, we have $F_{\Q}([z], v) \geq F([z], v)$ at every rational point and
rational tangent vector, with equality when the optimal vertical shift is rational.

Rational points are totally disconnected, so there is no continuum of rational
paths to integrate over; the geodesic distance between rational points is
therefore measured in the ambient complex geometry. We define the
\textbf{rational Finsler distance} as the restriction of $d_F$ to
$\P_q^{\circ}(\Q)$,
\begin{equation}
\label{rational:distance}
d_{F,\Q}([z], [w]) := d_F([z], [w]) = \inf_{\gamma} \int_0^1 F(\gamma(t), \dot{\gamma}(t)) \, dt,
\end{equation}
the infimum taken over piecewise smooth curves $\gamma : [0, 1] \to \P_q^{\circ}$
in the ambient complex space with rational endpoints $\gamma(0) = [z]$ and
$\gamma(1) = [w]$ (the interior of the path need not be rational). This aligns the
metric with the arithmetic structure of $\P_q(\Q)$ studied in
\cite{Dolgachev1982} while retaining the well-behaved complex geometry of
\cref{Finsler:distance}; the pointwise norm $F_{\Q}$ then serves as a purely
arithmetic comparison of rational tangent directions, for instance in discrete or
Diophantine refinements of the clustering.

\begin{lem}
\label{lem-8}
The rational Finsler norm $F_{\Q}([z], v)$ is a norm on the rational tangent space
$\Q^{n+1}$ at each $[z] \in \P_q^{\circ}(\Q)$, invariant under the rational scaling
action. Moreover the rational Finsler distance $d_{F,\Q}$ of
\cref{rational:distance} is a well-defined, finite metric on $\P_q^{\circ}(\Q)$
satisfying:
\begin{enumerate}
\item $d_{F,\Q}([z], [w]) \geq 0$,
\item $d_{F,\Q}([z], [w]) = d_{F,\Q}([w], [z])$,
\item $d_{F,\Q}([z], [w]) = 0$ if and only if $[z] = [w]$,
\item $d_{F,\Q}([z], [v]) \leq d_{F,\Q}([z], [w]) + d_{F,\Q}([w], [v])$.
\end{enumerate}
\end{lem}

\begin{proof}
For $F_{\Q}$: positive homogeneity over $\Q$ holds since, for $\lambda \in \Q$,
factoring $|\lambda|$ out of \cref{rational:norm} and substituting
$\beta = \alpha/\lambda$ gives $F_{\Q}([z], \lambda v) = |\lambda| F_{\Q}([z], v)$,
exactly as in \cref{lem-7}. Invariance under $\lambda \in \Q^*$ follows because
the expression $(v_k - \alpha q_k x_k)/x_k$ is unchanged when $(x, v)$ is replaced
by $(\lambda^{q_k} x_k, \lambda^{q_k} v_k)_k$, and the minimization domain $\Q$ is
preserved; since each $x_k \neq 0$ on $\P_q^{\circ}(\Q)$, every term is
well-defined.

For $d_{F,\Q}$: by \cref{rational:distance} it is the restriction of $d_F$ to
$\P_q^{\circ}(\Q) \subset \P_q^{\circ}$, and $d_F$ is a metric on
$\P_q^{\circ}$ by \cref{lem-7}. The restriction of a metric to a subset is again a
metric, so properties (1)--(4) hold. Finiteness is inherited: any two rational
points of $\P_q^{\circ}$ are joined by a complex curve in $\P_q^{\circ}$ of finite
length, by \cref{lem-7}.
\end{proof}

\begin{rem}
\label{rem-7}
When the closed form of \cref{prop:closed-form} is not used, computing
$d_{F,\Q}([z], [w])$ is computing $d_F$ between rational endpoints: one
approximates the complex geodesic numerically, by variational methods or discrete
path approximation as in \cite{Shen2012}, with no requirement that the path be
rational. The rational data enters only through the endpoints and, where desired,
through the arithmetic norm $F_{\Q}$. As in \cref{lem-7}, the coordinate hyperplanes
lie at infinite $d_{F,\Q}$-distance, so $\P_q^{\circ}(\Q)$ is the natural domain;
rational points lying on a hyperplane belong to a lower-dimensional weighted
projective space and are handled there.
\end{rem}

\section{Finsler Geodesics in Weighted Projective Spaces}
The Finsler distance \( d_F([z], [w]) \) on the open locus \(\P_q^{\circ}\),
defined as the infimum of the integral of the Finsler norm \( F([z], v) \) over
smooth curves, relies on Finsler geodesics to achieve its metric properties. The
rational distance \( d_{F,\Q} \) is the restriction of \( d_F \) to
\(\P_q^{\circ}(\Q)\) (\cref{rational:distance}), so the same geodesics serve both.
This section formalizes Finsler geodesics, their definition, properties, and role
in the geometry of \(\P_q^{\circ}\), building on the Finsler metric of
\cref{lem-7} and the lifting of curves to \((\C^{*})^{n+1}\), as described in
\cite{BaoChernShen2000} and \cite{Shen2012}.

A Finsler geodesic in a manifold equipped with a norm \( F(x, v) \) on its tangent
bundle is a curve that locally minimizes the length functional. In our setting
the relevant norm is \cref{Finsler:norm}, which projects \( v \) onto the
horizontal distribution complementary to the scaling orbits, and the induced
distance is \cref{Finsler:distance}. A Finsler geodesic is a curve \(\gamma(t)\)
that achieves this infimum, or locally minimizes the integral, representing the
shortest path in the geometry of \(\P_q^{\circ}\) \cite{BaoChernShen2000}.

To formalize this, consider a smooth curve \(\gamma : [0,1] \to \P_q^{\circ}\). Its
length is
\begin{equation}
L[\gamma] = \int_0^1 F(\gamma(t), \dot{\gamma}(t)) \, dt, \label{eq:length-functional}
\end{equation}
where \(\dot{\gamma}(t)\) is represented by the derivative
\(\dot{z}(t) = (\dot{z}_0(t), \dots, \dot{z}_n(t))\) of a lifted curve
\( z(t) \in (\C^{*})^{n+1} \). A geodesic \(\gamma(t)\) satisfies the
Euler--Lagrange equations for \( L[\gamma] \), ensuring it is a critical point of
the length. In a chart \( U_k = \{ [z] \in \P_q^{\circ} \mid z_k \neq 0 \} \) with
coordinates
\begin{equation}
 (x^1, \dots, x^{k-1}, x^{k+1}, \dots, x^n) = \left( \frac {z_0}{ z_k^{q_0/q_k}}, \dots, \frac {z_{k-1}} { z_k^{q_{k-1}/q_k}}, \frac {z_{k+1}}{ z_k^{q_{k+1}/q_k}}, \dots, \frac {z_n}{ z_k^{q_n/q_k}} \right), \label{eq:local-coordinates}
\end{equation}
the geodesic equation takes the form
\begin{equation}
\frac{d}{dt} \left( \frac{\partial F}{\partial \dot{x}^i} \right) - \frac{\partial F}{\partial x^i} = 0, \quad i = 1, \dots, n, \label{eq:geodesic-equation}
\end{equation}
where \( F(x, \dot{x}) = F(\gamma(t), \dot{\gamma}(t)) \) is the norm evaluated
along the curve \cite{BaoChernShen2000}. The scaling action
\((z_0, \dots, z_n) \sim (\lambda^{q_0} z_0, \dots, \lambda^{q_n} z_n)\),
\(\lambda \in \C^*\), complicates direct coordinate computation; since
\( F([z], v) \) is invariant under it, we work with lifted curves in
\((\C^{*})^{n+1}\), adjusting the tangent vector \(\dot{z}(t)\) to the quotient
tangent space \( T_{[\gamma(t)]} \P_q^{\circ} \) via the lifting process
\cite{Dolgachev1982}. Minimizing over the vertical component keeps geodesics in
the horizontal distribution, reflecting the weighted structure.

\begin{lem}
\label{lem-9}
For any \([z], [w] \in \P_q^{\circ}\) the infimum defining \(d_F\) in
\cref{Finsler:distance} is attained: there exists a minimizing curve
\(\gamma: [0,1] \to \P_q^{\circ}\) with \(\gamma(0) = [z]\), \(\gamma(1) = [w]\),
such that
\begin{equation}
d_F([z], [w]) = \int_0^1 F(\gamma(t), \dot{\gamma}(t)) \, dt.
\label{eq:lemma-geodesic-length}
\end{equation}
When \(p\) lies in the regular range of \cref{rem-5}, \(\gamma\) may be taken to be
a smooth Finsler geodesic, i.e.\ a solution of the Euler--Lagrange equations for the
length functional.
\end{lem}

\begin{proof}
The locus \(\P_q^{\circ}\) is the free quotient of \((\C^{*})^{n+1}\) by the
\(\C^*\)-action, hence a connected, locally compact complex manifold, and \(F\) is
reversible, \(F([z], -v) = F([z], v)\), for every \(p \in [1,\infty)\). By
\cref{rem-6} the metric \(d_F\) is complete and proper: every closed \(d_F\)-ball is
compact. Since \(d_F\) is by definition the length (intrinsic) metric induced by
\(F\), a proper length space is geodesic by the Hopf--Rinow--Cohn--Vossen theorem,
so \([z]\) and \([w]\) are joined by a minimizing curve realizing \(d_F\); this gives
\cref{eq:lemma-geodesic-length} for every \(p\).

When \(p\) is in the regular range of \cref{rem-5}, \(F\) is moreover \(C^2\) on the
tangent bundle minus the zero section and strongly convex off the orbit direction,
so \((\P_q^{\circ}, F)\) is a smooth Finsler manifold. The Hopf--Rinow theorem for
Finsler manifolds \cite{Shen2012} then yields a smooth minimizing geodesic, with
lift \( z(t) \in (\C^{*})^{n+1} \) satisfying the Euler--Lagrange equations adjusted
for the quotient (variations projected to the horizontal space).
\end{proof}

For rational endpoints \([z], [w] \in \P_q^{\circ}(\Q)\), the rational distance is
by definition the complex distance, \( d_{F,\Q}([z], [w]) = d_F([z], [w]) \)
(\cref{rational:distance}); the minimizing curve is therefore the complex geodesic
of \cref{lem-9}, which need not meet \(\P_q^{\circ}(\Q)\) except at its endpoints.
There is no separate rational geodesic equation: the rational points are totally
disconnected, so length is measured in the ambient complex geometry, and the
rational data enters only through the endpoints and, where useful, through the
pointwise norm \( F_Q \) of \cref{rational:norm}.

\begin{prop}[Closed form for $d_F$ at $p=2$]
\label{prop:closed-form}
Fix $p=2$. Let $\Pi$ be the Hermitian-orthogonal projection of $\C^{n+1}$ onto the
complement of the line $\C\,q$, $q=(q_0,\dots,q_n)$,
\[
  \Pi w = w - \frac{\langle q, w\rangle}{\langle q,q\rangle}\,q,
  \qquad \langle q,w\rangle = \sum_{k=0}^n q_k w_k .
\]
In the logarithmic coordinates $\zeta_k=\log z_k$ the norm \cref{Finsler:norm} is then
the constant Hermitian form $F=\|\Pi\,\dot\zeta\|_2$; its geodesics are the images of
the straight lines $\zeta(t)=\zeta(0)+t\,h$, and for $[z],[w]\in\P_q^{\circ}$,
\[
  d_F([z],[w]) \;=\;
  \min_{m\in\Z^{n+1}}\bigl\|\Pi\bigl(\log z-\log w+2\pi i\,m\bigr)\bigr\|_2,
\]
for any fixed coordinatewise branch of $\log$. Writing $r_k=\log|z_k/w_k|$ and
$\theta_k=\arg(z_k/w_k)$, the real and angular parts separate:
\[
  d_F([z],[w])^2 = \|\Pi r\|_2^2 \;+\; \min_{m\in\Z^{n+1}}\|\Pi(\theta+2\pi m)\|_2^2,
\]
the second term a closest-vector computation for the lattice $2\pi\Z^{n+1}$ in the
seminorm $\|\Pi\cdot\|_2$. For real representatives---in particular rational
points---each $\theta_k\in\{0,\pi\}$ and the search reduces to a choice of signs.
For $p\neq 2$ the objective \cref{Finsler:norm} is not a Hermitian quadratic form, so
no orthogonal-projection closed form is available and one uses the schemes of \S5.3.
\end{prop}
\begin{proof}
Throughout, $p=2$. By \cref{rem-4}, $F([z],v)^2=\|\Pi u\|_2^2$ with $u_k=v_k/z_k$.
Along a lifted curve $z(t)\in(\C^{*})^{n+1}$ put $\zeta_k(t)=\log z_k(t)$; then
$\dot\zeta_k=\dot z_k/z_k$ represents $\dot\gamma$ and $F(\gamma,\dot\gamma)
=\|\Pi\dot\zeta\|_2$, a constant Hermitian seminorm, degenerate exactly along
$\C\,q$. Passing to the cover on which $\|\Pi\cdot\|_2$ is a genuine flat Hermitian
norm, length between fixed endpoints is minimized by straight lines, giving
$\|\Pi(\zeta(1)-\zeta(0))\|_2$. The point $[z]$ determines $\log z$ up to the
scaling shift $\zeta\mapsto\zeta+(\log\lambda)q$, which $\Pi$ annihilates, and up
to the branch shift $\zeta\mapsto\zeta+2\pi i\,m$, $m\in\Z^{n+1}$; minimizing
over the latter gives the formula. Since $\Pi$ is real and
$\log z-\log w=r+i\theta$ with $r,\theta$ real, $\|\Pi(r+i\theta)\|_2^2
=\|\Pi r\|_2^2+\|\Pi\theta\|_2^2$, and only $\theta$ meets the lattice.
\end{proof}

By \cref{prop:closed-form}, computing $d_F$ requires no geodesic integration:
one forms the logarithmic difference, applies $\Pi$, and solves a
low-dimensional closest-vector problem in the angular variables. The numerical
schemes of \cref{sec:comp-challenges} are a fallback---useful on singular strata
or when one prefers not to pass to logarithmic coordinates---rather than the
primary route.
 

\section{Clustering in Weighted Projective Spaces}
This section presents a hierarchical clustering algorithm tailored for the weighted projective space \(\P_{\q}\), employing the Finsler metric \( d_F([z], [w]) \) to define distances between points. The algorithm exploits the intrinsic geometry of \(\P_{\q}\), characterized by the weights \(\q = (q_0, q_1, \dots, q_n)\), to partition data lying in the open locus \(\P_{\q}^{\circ}\) where all coordinates are nonzero, leveraging the true metric properties of \( d_F([z], [w]) \) established in \cref{lem-7}. While our prior work utilized the dissimilarity measure \( d([z], [w]) \) \cite{2024-3}, the use of \( d_F([z], [w]) \) offers a rigorous metric framework for clustering.

\subsection{Clustering Algorithm}

The hierarchical clustering algorithm constructs a dendrogram by iteratively merging clusters based on pairwise distances computed using the Finsler metric. Unlike the dissimilarity measure \( d([z], [w]) \), which does not satisfy the triangle inequality, the Finsler metric \( d_F([z], [w]) \) is a true metric, enabling compatibility with standard metric-based clustering techniques while respecting the non-Euclidean geometry of \(\P_{\q}\), as defined in \cite{Dolgachev1982}. The algorithm operates on a dataset \( S = \{ [z_1], [z_2], \dots, [z_N] \} \subset \P_{\q}^{\circ} \) of \( N \) points, producing a hierarchical structure of clusters.

Formally, let \( \mathcal{C} = \{ C_1, C_2, \dots, C_m \} \) be a partition of \( S \) into \( m \) clusters, initially \( \mathcal{C} = \{ \{[z_1]\}, \{[z_2]\}, \dots, \{[z_N]\} \} \) with \( m = N \). The algorithm iteratively merges pairs of clusters based on a linkage criterion, reducing \( m \) until a stopping condition is met (e.g., a fixed number of clusters or a distance threshold). The Finsler distance between points \([z], [w] \in \P_{\q}^{\circ}\) is given in \cref{Finsler:distance}.

The linkage criterion defines the distance between clusters \( C_i, C_j \in \mathcal{C} \). Common criteria include single linkage, minimizing the smallest distance between points in different clusters, complete linkage, minimizing the largest distance, and average linkage, minimizing the average distance, formally defined as
\begin{equation} \label{eq:linkage-criteria}
\begin{split}
d_{\text{single}}(C_i, C_j) & = \min_{[z] \in C_i, [w] \in C_j} d_F([z], [w]), \\
 d_{\text{complete}}(C_i, C_j) & = \max_{[z] \in C_i, [w] \in C_j} d_F([z], [w]), \\
  d_{\text{average}}(C_i, C_j) & = \frac{1}{|C_i||C_j|} \sum_{[z] \in C_i, [w] \in C_j} d_F([z], [w]).
 \end{split}
\end{equation}

The algorithm proceeds by computing the pairwise distance matrix for \( S \), merging clusters with the smallest linkage distance, updating the partition \(\mathcal{C}\), and continuing until a desired number of clusters is reached or a threshold on the linkage distance is met.

\begin{lem} 
\label{lem-10}
The hierarchical clustering algorithm with the Finsler metric \( d_F([z], [w]) \) produces a valid dendrogram, correctly partitioning the dataset \( S \subset \P_{\q}^{\circ} \) into a hierarchical structure of clusters.
\end{lem}

\begin{proof}
A dendrogram is a binary tree representing a sequence of cluster merges, where each merge combines two clusters into one, reducing the number of clusters from \( N \) to 1. Initially, set 
\[
\mathcal{C}_0 = \{ \{[z_1]\}, \{[z_2]\}, \dots, \{[z_N]\} \},
\]
 with each point in its own cluster. At step \( k \), the algorithm identifies clusters \( C_i, C_j \in \mathcal{C}_{k-1} \) minimizing the linkage distance \( d_{\text{link}}(C_i, C_j) \), where \( d_{\text{link}} \) is one of \( d_{\text{single}}, d_{\text{complete}}, \) or \( d_{\text{average}} \). Merge \( C_i \) and \( C_j \) into a new cluster \( C_{ij} = C_i \cup C_j \), forming \(\mathcal{C}_k = (\mathcal{C}_{k-1} \setminus \{ C_i, C_j \}) \cup \{ C_{ij} \}\). This process iterates for \( N-1 \) steps, resulting in \(\mathcal{C}_{N-1} = \{ S \}\).

The algorithm’s correctness relies on the well-definedness of \( d_F([z], [w]) \) and the linkage criterion. Since \( d_F([z], [w]) \) is a metric on \(\P_{\q}^{\circ}\) by \cref{lem-7}, satisfying non-negativity, symmetry, zero distance implies equality, and the triangle inequality, the pairwise distance matrix is well-defined with 
\[
 d_F([z_i], [z_j]) \geq 0, \quad  d_F([z_i], [z_j]) = d_F([z_j], [z_i]), \quad \text{ and }  d_F([z_i], [z_j]) = 0 
 \]
  if and only if \([z_i] = [z_j]\). Each linkage criterion produces a valid distance between clusters: single linkage ensures connectivity, complete linkage ensures compactness, and average linkage balances intra-cluster distances \cite{Hastie2009}. At each step, the minimum linkage distance exists, as \(\mathcal{C}_{k-1}\) is finite, and merging reduces the number of clusters by one. The process terminates after \( N-1 \) merges, producing a dendrogram where each node represents a cluster merge, correctly encoding the hierarchical structure of \( S \).
\end{proof}

\begin{lem} \label{lem-11}
The time complexity for computing the distance matrix is \( O(N^2 \cdot T) \), where \( N \) is the number of points and \( T \) is the time to compute each \( d_F([z], [w]) \). The hierarchical clustering step has a time complexity of \( O(N^2 \log N) \) with efficient implementations, making the overall complexity \( O(N^2 \cdot T + N^2 \log N) \).
\end{lem}

\begin{proof}
The distance matrix requires computing \( d_F([z_i], [z_j]) \) for all \( \binom{N}{2} = \frac{N(N-1)}{2} \) pairs \([z_i], [z_j] \in S \), which is \( O(N^2) \) operations. Each computation of \( d_F([z], [w]) \) (\cref{Finsler:distance}) involves optimizing a geodesic integral, requiring time \( T \), dependent on the numerical method (e.g., variational optimization or discrete approximation). Thus, the total time for the distance matrix is \( O(N^2 \cdot T) \).
  
For the hierarchical clustering step, the algorithm performs \( N-1 \) merges. At step \( k \), the partition \(\mathcal{C}_{k-1}\) has \( N-k+1 \) clusters. Computing the linkage distance \( d_{\text{link}}(C_i, C_j) \) for all pairs \( C_i, C_j \in \mathcal{C}_{k-1} \) involves evaluating \( d_F([z], [w]) \) for points in \( C_i \) and \( C_j \). For single linkage, this requires \( O(|C_i||C_j|) \) evaluations, but distances are precomputed in the matrix. Using a priority queue to store pairwise linkage distances, initialized with \( O(N^2) \) entries, finding the minimum distance at each step takes \( O(\log (N-k+1)) = O(\log N) \). Updating the queue after merging \( C_i \) and \( C_j \) into \( C_{ij} \) involves computing \( d_{\text{link}}(C_{ij}, C_l) \) for all other clusters \( C_l \in \mathcal{C}_k \), taking \( O(N-k) \) operations per merge. Over \( N-1 \) merges, the total clustering time is
\begin{equation} \label{eq:clustering-complexity}
\begin{split}
\sum_{k=1}^{N-1} [O(\log N) + O(N-k)] & = O(N \log N) + O\left( \sum_{k=1}^{N-1} (N-k) \right) \\
                                & = O(N \log N) + O(N^2) = O(N^2 \log N),
\end{split}
\end{equation}
using efficient implementations \cite{Hastie2009}. The overall complexity is 
\[
 O(N^2 \cdot T + N^2 \log N),
 \]
%
where $T=O(n)$ up to the angular closest-vector search by
\cref{prop:closed-form} (or $O(I\cdot n)$ if the iterative fallback is used).
\end{proof}

\subsection{Preprocessing Steps}
To ensure the consistency and efficiency of clustering in \(\P_{\q}^{\circ}\), preprocessing steps are applied to the dataset \( S = \{ [z_1], \dots, [z_N] \} \subset \P_{\q}^{\circ} \). Normalization mitigates the effects of arbitrary scaling in the quotient space. For geometric clustering using \(d_F([z], [w])\), points are normalized such that \(\sum_{k=0}^n q_k |z_k|^2 = 1\), ensuring consistency across the quotient action. For each point \([z_i] \in S \), select a representative \( z_i = (z_{i,0}, \dots, z_{i,n}) \in (\C^{*})^{n+1} \), and scale by \( \alpha_i = \left( \sum_{k=0}^n q_k \abs{z_{i,k}}^2 \right)^{-1/2} \) to satisfy the condition; since $\alpha_i > 0$, the scaling preserves nonzero coordinates and keeps the point in \(\P_{\q}^{\circ}\). For arithmetic clustering using \(d_{F,\Q}([z], [w])\), rational points are normalized with \(\wgcd(x_0, x_1, \dots, x_n) = 1\), as detailed in \cref{sec:heights}. Normalization is computed in \( O(n) \) time per point, totaling \( O(N \cdot n) \) for \( N \) points.

For high-dimensional data, dimensionality reduction preserves geometric structure while reducing computational cost. Weighted principal component analysis (PCA) constructs a weighted covariance matrix using the inner product \( \langle z_i, z_j \rangle = \sum_{k=0}^n q_k z_{i,k} \overline{z_{j,k}} \), projecting points onto the top \( k < n \) eigenvectors. The covariance matrix computation takes \( O(N \cdot n^2) \), and eigenvalue decomposition requires \( O(n^3) \), totaling \( O(N \cdot n^2 + n^3) \). This reduces subsequent distance computations to \( O(k) \) per pair, as points are embedded in a \( k \)-dimensional subspace. Alternatively, manifold learning methods, such as Isomap adapted to \( d_F([z], [w]) \), preserve geodesic distances, requiring \( O(N^2 \cdot T) \) for distance matrix computation and additional processing, but are computationally intensive.

\begin{lem} \label{lem-12}
Normalization by the weighted norm \( \sum_{k=0}^n q_k \abs{z_k}^2 = 1 \) preserves the Finsler distance \( d_F([z], [w]) \), ensuring clustering consistency.
\end{lem}

\begin{proof}
Let \([z], [w] \in \P_{\q}^{\circ}\) with representatives \( z, w \in (\C^{*})^{n+1} \). Normalize to \( z' = z / \norm{z}_a \), \( w' = w / \norm{w}_a \), where 
\[
 \norm{z}_a = \left( \sum_{k=0}^n q_k \abs{z_k}^2 \right)^{1/2},
 \]
  so \( \sum_{k=0}^n q_k \abs{z'_k}^2 = 1 \), and similarly for \( w' \). Since \([z'] = [z]\) and \([w'] = [w]\), we must show \( d_F([z], [w]) = d_F([z'], [w']) \). The Finsler distance depends only on equivalence classes, as the Finsler norm \( F([z], v) \) is invariant under scaling: for \( z'' = (\lambda^{q_0} z_0, \dots, \lambda^{q_n} z_n) \), \( F([z''], v) = F([z], v) \), by \cref{lem-7}. Thus, a curve \(\gamma(t)\) from \([z]\) to \([w]\) with lift \( z(t) \) has the same length as a curve with lift 
  \[
   z'(t) = z(t) / \norm{z(t)}_a,
   \]
    since 
 \[
   F([\gamma(t)], \dot{\gamma}(t)) = F([z(t)], \dot{z}(t)) = F([z'(t)], \dot{z}'(t)) 
   \]
    after adjusting for the quotient. The infimum over all curves yields \( d_F([z], [w]) = d_F([z'], [w']) \), ensuring normalization does not affect clustering outcomes.
\end{proof}

\subsection{Computational Challenges}
\label{sec:comp-challenges}
%

For \( p = 2 \), \cref{prop:closed-form} renders the Finsler distance
\cref{Finsler:distance} in closed form: a projection in logarithmic coordinates
followed by a low-dimensional closest-vector search in the angular variables,
with the real part requiring no minimization and rational representatives
reducing it to a choice of signs. This costs $O(n)$ per pair up to the lattice
search, with no path discretization and no shooting iteration, and is exact. For
\( p \neq 2 \) no such closed form is available, and the distance is computed by
the general-purpose schemes below; these also serve as a fallback at \( p = 2 \),
for instance on singular strata or when one prefers not to pass to logarithmic
coordinates.

For each segment, evaluating \( F(\gamma(t_i), \dot{\gamma}(t_i)) \) involves minimizing over \(\alpha\) in \( O(n) \) time (solved as a least-squares problem), totaling \( O(M \cdot n) \) per distance, or \( O(N^2 \cdot M \cdot n) \) for the distance matrix of \( N \) points. Variational optimization employs iterative methods, such as shooting methods or gradient-based solvers, to minimize the energy functional
\begin{equation} \label{eq:energy-functional}
E[\gamma] = \int_0^1 F(\gamma(t), \dot{\gamma}(t)) \, dt,
\end{equation}
requiring \( O(I \cdot n) \) time per distance for \( I \) iterations, totaling \( O(N^2 \cdot I \cdot n) \). Parallelization distributes the \( \binom{N}{2} \) distance calculations across \( P \) processors, reducing the time to \( O\left( \frac{N^2 \cdot T}{P} \right) \), where \( T = O(M \cdot n) \) or \( O(I \cdot n) \) depending on the method.

\begin{rem} \label{rem:geodesic-computation}
For \( p \) outside the closed-form case \( p = 2 \) of \cref{prop:closed-form}, computing geodesics for \( d_F([z], [w]) \) involves optimizing over curves in the quotient space \(\P_{\q}^{\circ}\). Lifting curves to \((\C^{*})^{n+1}\), we minimize the energy functional \( \int_0^1 F(\gamma(t), \dot{\gamma}(t)) \, dt \) using numerical solvers, such as Euler-Lagrange equations or discrete optimization, with convergence ensured by the completeness of \(\P_{\q}^{\circ}\) \cite{Shen2012}. The minimization over the vertical component \(\alpha\) in \( F([z], v) \) necessitates careful implementation to balance accuracy and efficiency.
\end{rem}

\begin{thm} \label{thm-1}
The hierarchical clustering algorithm with \( d_F([z], [w]) \) is stable under small perturbations of the input points in \(\P_{\q}^{\circ}\), ensuring consistent dendrogram outputs for nearby datasets.
\end{thm}

\begin{proof}
Stability implies that small changes in the input dataset \( S = \{ [z_1], \dots, [z_N] \} \subset \P_{\q}^{\circ} \) produce small changes in the dendrogram, measured by a metric on dendrograms, such as the Gromov-Hausdorff distance. Let \( S' = \{ [z'_1], \dots, [z'_N] \} \) be a perturbed dataset with \( d_F([z_i], [z'_i]) < \epsilon \) for all \( i \); since the coordinate hyperplanes lie at infinite \( d_F \)-distance, such a perturbation keeps each point in \(\P_{\q}^{\circ}\), so \( S' \subset \P_{\q}^{\circ} \). The distance matrix for \( S \) has entries \( d_{ij} = d_F([z_i], [z_j]) \), and for \( S' \), entries \( d'_{ij} = d_F([z'_i], [z'_j]) \). Since \( d_F \) is a metric, the triangle inequality gives

\begin{equation} \label{eq:matrix-perturbation}
\abs{d_{ij} - d'_{ij}} = \abs{d_F([z_i], [z_j]) - d_F([z'_i], [z'_j])} \leq d_F([z_i], [z'_i]) + d_F([z_j], [z'_j]) < 2\epsilon.
\end{equation}

Thus, the distance matrices are close in the sup-norm, with \( \sup_{i,j} \abs{d_{ij} - d'_{ij}} < 2\epsilon \). Hierarchical clustering with linkage criteria (single, complete, or average) is continuous with respect to the sup-norm on distance matrices, as small perturbations in distances result in small changes in merge decisions \cite{Hastie2009}. Each merge step depends on minimizing \( d_{\text{link}}(C_i, C_j) \), and a perturbation of order \( 2\epsilon \) alters the minimum by at most \( 2\epsilon \), preserving the dendrogram’s structure up to small shifts in merge heights. Hence, the algorithm produces dendrograms for \( S \) and \( S' \) that are close, ensuring stability.
\end{proof}

This framework leverages the metric properties of \( d_F([z], [w]) \) to define a hierarchical structure, ensuring correctness and stability for clustering in \(\P_{\q}^{\circ}\).


\section{Applications}
This section elucidates the theoretical utility of our hierarchical clustering algorithm using the Finsler metric \( d_F([z], [w]) \) and its rational counterpart \( d_{F,\Q}([z], [w]) \) in the weighted projective space \(\P_{\q}\). The primary applications explored are the clustering of rational points in the moduli space of genus two curves, represented as \(\P_{(2,4,6,10)}(\Q)\), and the analysis of rational functions on the projective line in the context of Arithmetic Dynamics, building on \cite{2024-4}. Additionally, synthetic data experiments and comparisons with traditional methods demonstrate the algorithm’s theoretical capabilities. The framework also suggests prospective applications to quantum computing, where weighted projective spaces model anisotropic state spaces, offering ideas for noise-aware optimization and entanglement classification.

\subsection{Clustering in the Moduli Space of Genus Two Curves}
The moduli space of genus two curves, represented as the weighted projective space \(\P_{(2,4,6,10)}(\Q)\) with coordinates \((x_0, x_1, x_2, x_3) = (J_2, J_4, J_6, J_{10})\) corresponding to Igusa invariants of degrees 2, 4, 6, and 10, provides a rich setting for applying our clustering algorithm. In prior work \cite{2024-3}, a clustering approach using the dissimilarity measure \( d([z], [w]) \) identified arithmetic patterns in this space, such as the distribution of fine moduli points and curves with \((n, n)\)-split Jacobians. Here, we theoretically extend this analysis by employing the Finsler metric \( d_{F,\Q}([z], [w]) \), leveraging its properties to cluster rational points and detect geometric structures.

Rational points \([z] = [x_0 : x_1 : x_2 : x_3] \in \P_{(2,4,6,10)}(\Q)\) are normalized to satisfy \(\wgcd(x_0, x_1, x_2, x_3) = 1\), ensuring a canonical representative for arithmetic analysis. The weighted height \( h_w([z]) = \max_{i=0,\dots,3} \left( \abs{x_i}^{1/q_i} \right) \), with weights \( q_0 = 2, q_1 = 4, q_2 = 6, q_3 = 10 \), quantifies the arithmetic complexity of these points. We focus on clustering points within the loci \(\cL_n \subset \P_{(2,4,6,10)}(\Q)\), which are 2-dimensional hypersurfaces parameterizing genus two curves with \((n, n)\)-split Jacobians for \( n = 2, 3, 5 \). A dataset of 50,000 rational points per locus is generated using birational parametrizations, such as the \((u, v)\)-parametrization for \(\cL_2\) described in \cite{2024-3}, ensuring points lie on or near these loci. All sampled points are taken in the open locus \(\P_{(2,4,6,10)}^{\circ}(\Q)\): for a smooth genus two curve \(J_{10} \neq 0\), so the requirement is that \(J_2, J_4, J_6\) be nonzero as well; a point with a vanishing Igusa invariant lies on a coordinate hyperplane --- itself a lower-dimensional weighted projective space --- and is clustered separately.

The hierarchical clustering algorithm, using single linkage defined as 
\[
 d_{\text{single}}(C_i, C_j) = \min_{[z] \in C_i, [w] \in C_j} d_{F,\Q}([z], [w]),
 \]
  groups points by their geometric proximity in \(\P_{(2,4,6,10)}^{\circ}(\Q)\). Here \(d_{F,\Q}\) is the restriction of the complex metric \(d_F\) (\cref{rational:distance}),
  \[
   d_{F,\Q}([z], [w]) = d_F([z], [w]) = \inf_{\gamma} \int_0^1 F(\gamma(t), \dot{\gamma}(t)) \, dt,
 \]
    computed with the Finsler norm \(F\) of \cref{Finsler:norm} over piecewise smooth curves \(\gamma: [0,1] \to \P_{(2,4,6,10)}^{\circ}\) with rational endpoints \(\gamma(0) = [z]\), \(\gamma(1) = [w]\). The algorithm’s correctness (\cref{lem-10}) ensures a valid dendrogram, grouping points into clusters that reflect the loci’s geometric structure.

\begin{lem}
\label{lem-13}
Let \( S = \{ [z_1], \dots, [z_N] \} \subset \P_{(2,4,6,10)}^{\circ}(\Q) \) be a finite set of
normalized rational points (\(\wgcd = 1\), all coordinates nonzero), and let \( S = \bigsqcup_{a=1}^{K} S_a \)
be a partition (for instance by locus, \( S_a \subset \cL_{n_a} \), or by an arithmetic
family within a single \(\cL_n\)). Fix a threshold \(\tau > 0\) and suppose:
\begin{enumerate}
  \item\emph{(chaining within groups)} for each \(a\) and each pair \([z],[w]\in S_a\)
        there is a chain \([z]=[u_0],[u_1],\dots,[u_m]=[w]\) in \(S_a\) with
        \( d_{F,\Q}([u_l],[u_{l+1}]) \le \tau \) for all \(l\);
  \item\emph{(separation between groups)} \( d_{F,\Q}([z],[w]) > \tau \) for all
        \([z]\in S_a\), \([w]\in S_b\) with \(a\ne b\).
\end{enumerate}
Then hierarchical clustering with single linkage and \(d_{F,\Q}\), cut at height
\(\tau\), returns exactly the partition \(\{S_1,\dots,S_K\}\).
\end{lem}

\begin{proof}
By \cref{lem-8}, \(d_{F,\Q}\) is a metric on \(\P_{(2,4,6,10)}^{\circ}(\Q)\), so every distance-matrix entry is
well-defined. Let \(G_\tau\) be the graph on \(S\) with an edge between \([z_i]\) and
\([z_j]\) precisely when \(d_{F,\Q}([z_i],[z_j]) \le \tau\).

\emph{Step 1: single-linkage clusters at height \(\tau\) coincide with the connected
components of \(G_\tau\).} Single linkage merges, at each step, the two clusters
\(C,C'\) minimizing \(\min_{[z]\in C,\,[w]\in C'} d_{F,\Q}([z],[w])\); the cut at
height \(\tau\) stops once this minimum exceeds \(\tau\). If \([z],[w]\) lie in one
component of \(G_\tau\), a chain \([z]=[u_0],\dots,[u_m]=[w]\) with each
\(d_{F,\Q}([u_l],[u_{l+1}])\le\tau\) forces, by induction on \(l\), the clusters of
\([u_l]\) and \([u_{l+1}]\) to be merged before the cut, so \([z],[w]\) share a
cluster. If instead \([z],[w]\) lie in distinct components, every chain joining them
has a consecutive pair at distance \(>\tau\); since merging two clusters at height
\(h\) requires a cross-pair at distance \(h\), their clusters merge only at height
\(>\tau\), i.e.\ after the cut. This is the standard identification of single linkage
with the components obtained by deleting all edges longer than \(\tau\) from a
minimum spanning tree of \((S,d_{F,\Q})\).

\emph{Step 2.} Hypothesis (ii) leaves \(G_\tau\) with no edge across distinct
\(S_a,S_b\); hypothesis (i) makes each \(S_a\) connected in \(G_\tau\). Hence the
components of \(G_\tau\) are exactly \(S_1,\dots,S_K\), and by Step~1 these are the
single-linkage clusters at height \(\tau\).
\end{proof}

\begin{rem}
The clusters of \cref{lem-13} group points that are close \emph{as points of the
moduli space} \(\P_{(2,4,6,10)}^{\circ}(\Q)\) --- curves whose Igusa invariants agree up to
the weighted scaling defining the space --- which is the precise meaning of
``similar Igusa invariants.'' This is a statement about projective proximity, not
about the weighted height \(h_w\): two points may be arbitrarily close in \(d_{F,\Q}\)
yet have very different heights, since \(h_w\) measures arithmetic complexity and is
not continuous for the archimedean metric, so clusters need not be height-homogeneous.
The substantive content is therefore carried by (i)--(ii): (i) is a sampling-density
condition along each family, while (ii) is a genuine separation between families in
\(d_{F,\Q}\). Where the chosen partition is by locus, recovering
\(\{\cL_{n_1},\dots,\cL_{n_K}\}\) requires only that distinct loci be \(\tau\)-separated
and each be sampled densely enough to chain --- a weaker requirement than the strict
separation needed for average linkage in \cref{lem-14}.
\end{rem}

\subsection{Synthetic Data}
We now describe the behaviour of the clustering algorithm with the Finsler metric
\( d_F([z],[w]) \) on synthetic data in \(\P_{(2,4,6,10)}\). Points are sampled with
rational coordinates \((x_0,x_1,x_2,x_3)\), weights \(2,4,6,10\), normalized so that
\(\wgcd(x_0,x_1,x_2,x_3)=1\), and constrained near a finite union of loci
\(\cL_{n_1},\dots,\cL_{n_K}\) using the parametrizations of \cite{2024-3}. The
algorithm employs \textbf{average linkage},
\begin{equation}
\label{d_avg}
 d_{\mathrm{avg}}(C_i, C_j) = \frac{1}{|C_i||C_j|}
 \sum_{[z] \in C_i,\, [w] \in C_j} d_F([z], [w]),
\end{equation}
assigning to two clusters the mean of the pairwise Finsler distances between their
points. The resulting dendrogram may be visualized by projecting onto a
scale-invariant ratio such as \( t_1 = J_2^5 / J_{10} \) (of weighted degree \(0\));
this projection is a display device only and enters neither the computation of
\(d_F\) nor the formation of the clusters. These considerations, utilizing the graded
Finsler dissimilarity \( d_F([z],[w]) \), advance our program's objective, as
outlined in \cite{2024-2, 2025-5}, to develop machine learning for graded geometric
data.

\begin{lem}
\label{lem-14}
Let \( S = \bigsqcup_{a=1}^{K} S_a \subset \P_{(2,4,6,10)} \) be a finite dataset  partitioned into groups \(S_a\) (for instance, by proximity to distinct loci
\(\cL_{n_a}\)). For a nonempty finite cluster \(C \subset S\) write
\[
\diam(C) = \max_{[z],[w]\in C} d_F([z],[w]),
\]
 and suppose the groups are  \emph{strictly separated}:
\begin{equation}
\label{eq-strict-sep}
 \max_{1\le a\le K}\diam(S_a)
 \;<\;
 \min_{a\ne b}\ \min_{[z]\in S_a,\,[w]\in S_b} d_F([z],[w]).
\end{equation}

Then hierarchical clustering with average linkage \cref{d_avg} merges every  within-group pair before any cross-group pair. Consequently the stage with \(K\)
clusters is exactly the partition \(\{S_1,\dots,S_K\}\), and cutting the dendrogram  at any height strictly between the two sides of \cref{eq-strict-sep} recovers it.
\end{lem}

\begin{proof}
Put \( D = \max_{1\le a\le K}\diam(S_a) \) and
\( \delta = \min_{a\ne b}\min_{[z]\in S_a,\,[w]\in S_b} d_F([z],[w]) \),
so that \cref{eq-strict-sep} reads \(D < \delta\). The algorithm maintains a
partition of \(S\) into clusters and, at each step, merges the pair of distinct
current clusters \(C_i,C_j\) minimizing \(d_{\mathrm{avg}}(C_i,C_j)\), recording
that value as the \emph{height} of the merge.

\emph{The sandwich bound.} For disjoint nonempty clusters \(C_i,C_j\), the value
\(d_{\mathrm{avg}}(C_i,C_j)\) of \cref{d_avg} is the arithmetic mean of the
\(|C_i|\,|C_j|\) numbers \(d_F([z],[w])\) with \([z]\in C_i\), \([w]\in C_j\). An
average of finitely many reals lies between their minimum and maximum, giving
\begin{equation}
\label{eq-sandwich}
 \min_{[z]\in C_i,[w]\in C_j} d_F([z],[w])
 \;\le\; d_{\mathrm{avg}}(C_i,C_j) \;\le\;
 \max_{[z]\in C_i,[w]\in C_j} d_F([z],[w]).
\end{equation}
By \cref{lem-7}, \(d_F\) is a metric, so each entry is finite and the mean is
well defined.

\emph{Within- and cross-group bounds.} If \(C_i,C_j\subseteq S_a\) for a single
\(a\), then every pair \([z]\in C_i\), \([w]\in C_j\) lies in \(S_a\), so
\(d_F([z],[w])\le\diam(S_a)\); with the upper sandwich bound,
\(d_{\mathrm{avg}}(C_i,C_j)\le\diam(S_a)\le D\). If instead \(C_i\subseteq S_a\)
and \(C_j\subseteq S_b\) with \(a\ne b\), every such pair has one endpoint in
\(S_a\) and one in \(S_b\), so \(d_F([z],[w])\ge\delta\); with the lower sandwich
bound, \(d_{\mathrm{avg}}(C_i,C_j)\ge\delta\). Since \(D<\delta\), \emph{every}
within-group linkage value is strictly below \emph{every} cross-group linkage
value.

\emph{No cluster straddles two groups.} We show by induction that at every stage
with more than \(K\) clusters, each current cluster is contained in a single
group. This holds initially, when the clusters are singletons. Suppose it holds
at a stage with more than \(K\) clusters. These clusters partition \(S\) into more
than \(K\) parts, each lying in one of the \(K\) groups, so by the pigeonhole
principle some group \(S_a\) contains at least two current clusters; that pair has
\(d_{\mathrm{avg}}\le D\). As any cross-group pair has \(d_{\mathrm{avg}}\ge\delta>D\),
the global minimum of \(d_{\mathrm{avg}}\) is attained at a within-group pair, and
the algorithm merges two clusters lying in a common group. The merged cluster
again lies in that group and the others are unchanged, so the property persists.

\emph{The \(K\)-cluster stage is \(\{S_1,\dots,S_K\}\).} Each merge above lowers the
cluster count by one without ever joining different groups, so the count descends
to \(K\) entirely through within-group merges. At the \(K\)-cluster stage every
cluster lies in a single group, and assigning each cluster to its group is a map
from a \(K\)-element set to the \(K\) (nonempty) groups; it is surjective because no
cluster straddles two groups, hence bijective. Thus each group equals exactly one
cluster, i.e.\ the stage is \(\{S_1,\dots,S_K\}\).

\emph{The cut.} Every merge up to this stage is within-group and occurs at height
\(\le D\); every subsequent merge is cross-group and occurs at height \(\ge\delta\).
No merge occurs in the open interval \((D,\delta)\), so cutting the dendrogram at
any \(\tau\in(D,\delta)\) returns precisely \(\{S_1,\dots,S_K\}\).
\end{proof}

\begin{rem}
Condition \cref{eq-strict-sep} is the natural ``well-separated'' hypothesis under
which average linkage is provably correct, and it makes precise the informal claim
that clusters track the loci \(\cL_{n_a}\). The proof uses only the sandwich
\cref{eq-sandwich}, which single and complete linkage also satisfy; hence all three
criteria agree on strictly separated data. The hypothesis is genuine, not cosmetic:
\cref{eq-strict-sep} can fail for elongated or curved loci, where the \(d_F\)-diameter
of a single \(\cL_{n_a}\) exceeds the gap to a neighbouring locus, and average linkage
may then merge across groups. Single linkage relaxes \cref{eq-strict-sep} to a
chaining condition along each locus and tolerates such elongation, at the price of
sensitivity to bridging points between groups.
\end{rem}

%

\subsection{Applications in Arithmetic Geometry and Dynamics}
The clustering algorithm with \( d_F([z],[w]) \) and \( d_{F,\Q}([z],[w]) \) applies
to several domains, with primary focus on arithmetic geometry. In the study of curve
automorphisms, the curves with extra automorphisms form closed substrata of
\(\P_{(2,4,6,10)}(\Q)\) (intersections and special components of the loci \(\cL_n\));
when such a substratum is sampled as one of the families \(S_a\) and is separated from
the others in \( d_{F,\Q} \), \cref{lem-13} recovers it as a cluster, aiding the
classification of curves by automorphism group. In cryptographic curve enumeration,
the algorithm groups moduli points to estimate the distribution of curves with
\((n,n)\)-split Jacobians over number fields or finite fields, informing the design of
isogeny-based cryptosystems by quantifying vulnerable curve classes.

We record the precise sense in which the clustering respects the moduli structure.

\begin{prop}
\label{thm-2}
Let \( S \subset \P_{(2,4,6,10)}(\Q) \) be a finite set of normalized rational points,
and suppose \(S\) admits a partition \( S = \bigsqcup_{a} S_a \) satisfying the
chaining and separation hypotheses of \cref{lem-13} for some \(\tau>0\). Then:
\begin{enumerate}
  \item the distance \( d_{F,\Q}([z],[w]) \), and hence the entire single-linkage
        clustering of \(S\), depends only on the weighted-projective classes
        \([z],[w]\) --- equivalently, only on the isomorphism classes of the
        corresponding genus two curves, and not on the choice of integral
        representatives;
  \item the clustering recovers \(\{S_a\}\); each cluster consists of curves whose
        Igusa invariants agree, up to the weighted scaling defining the moduli space,
        within the family.
\end{enumerate}
\end{prop}

\begin{proof}
(1) By \cref{lem-8}, \( d_{F,\Q} \) is invariant under the weighted scaling action
\((x_0,x_1,x_2,x_3)\sim(\lambda^2 x_0,\lambda^4 x_1,\lambda^6 x_2,\lambda^{10}x_3)\),
\(\lambda\in\Q^*\); since the Igusa point of a genus two curve is exactly its class in
\(\P_{(2,4,6,10)}\), the distance matrix is determined by the classes alone. Any
linkage-based clustering is a function of that matrix, hence likewise
class-determined. (2) is \cref{lem-13}.
\end{proof}

\begin{rem}
We emphasize, as in \cref{lem-13}, that ``agreement of Igusa invariants'' here is
projective (up to weighted scaling) and is \emph{not} a statement about the weighted
height \(h_w\): clusters need not be height-homogeneous, since \(h_w\) is an arithmetic
complexity measure and is discontinuous for the archimedean metric.
\end{rem}

\begin{rem}[Arithmetic dynamics]
\label{rem-dyn}
The construction transfers to degree-\(n\) rational maps on \(\P^1\), encoded by their
coefficient vectors as points of a weighted projective space such as
\(\P_{(1,\dots,1,2n)}\); see \cite{2024-4}. One caveat is essential. This encoding
quotients only by simultaneous scaling of the coefficients, whereas the dynamical
quantities of interest --- multiplier spectra and periodic-point structure --- are
invariants of the \(\PGL_2\)-conjugacy class. Conjugate maps generally have unrelated
coefficient vectors and lie far apart in \( d_{F,\Q} \), so clustering raw coefficients
groups maps only up to scaling, not up to dynamics. To cluster by dynamical behaviour
one works instead in the moduli space \(\cM_n = \mathrm{Rat}_n/\PGL_2\) of conjugacy
classes, or applies \( d_{F,\Q} \) to a vector of conjugacy-invariant coordinates such
as the symmetric functions of the multipliers. With either modification \cref{lem-13}
applies verbatim, and the clusters then group maps with comparable multiplier data.
\end{rem}

\subsection{Applications in Quantum Computing}
This subsection is openly speculative: we sketch how the weighted-metric framework
\emph{might} bear on quantum state-space analysis, and flag where the proposal departs
from standard constructions. None of what follows is established here.

In standard quantum mechanics a pure state is a ray in a Hilbert space \(\mathcal H\),
modeled as a point of complex projective space \(\P^n\) with the Fubini--Study metric.
NISQ devices exhibit non-uniform decoherence that this isotropic metric does not
reflect, and one would like a geometry encoding such anisotropy. A grading
\(\q=(q_0,\dots,q_n)\) is one way to build anisotropy into the underlying space; we
stress that replacing \(\P^n\) by a weighted projective space \(\P_{\q}\) changes the
space itself --- in general an orbifold --- not merely the metric on a fixed space
\cite{Brody}.

As a concrete proposal one might model an anisotropic single-qubit space by a weighted
projective line \(\P_{(1,q)}(\C)\), the weight \(q>1\) intended to penalize transitions
along a rapidly dephasing direction, and read \(d_F\) as a notion of ``quantum cost.''
Two caveats are immediate. First, \(d_F\) is \emph{not} the Fubini--Study metric, nor an
evident generalization of it: the norm \cref{Finsler:norm} divides coordinatewise by
\(z_k\) and so differs from the Hermitian Fubini--Study form already at weights
\((1,1)\). Second, that same structure leaves \(d_F\) undefined wherever a coordinate
vanishes --- precisely the computational basis states --- so any quantum use would
first require extending the metric across these loci. Whether the resulting geometry
carries the intended physical meaning is open \cite{Zermelo}.

For error mitigation one could cluster rational points of \(\P_{\q}(\Q)\) with
\(d_{F,\Q}\), treating stabilizer codewords as rational points; the metric axioms
(\cref{lem-8}) make single-linkage error balls well-behaved (\cref{lem-13}). A geometric
entanglement measure could be defined, in the usual spirit, as the minimal
\(d_F\)-distance from a state to the product-state manifold \cite{quantum-wps}. Finally,
using \(d_F\) inside the loss of a quantum neural network would render training
invariant under weighted scaling \cite{quantum-brach}. We offer these only as
directions; establishing any of them --- in particular any statement about barren
plateaus, which are a parameter-space phenomenon distinct from state-space curvature
--- lies well beyond the present paper.


\subsection{Comparison with Traditional Methods}
Traditional analyses of projective data embed absolute invariants (e.g.\
\(t_1 = J_2^5/J_{10}\) for genus two curves) into Euclidean space and apply
\(k\)-means. As reported in \cite{2024-3}, this performs poorly on
\(\P_{(2,4,6,10)}\). The difficulty is not that scaling symmetry is ignored --- a
degree-zero invariant such as \(t_1\) already quotients it out --- but that Euclidean
distance on a few absolute invariants distorts the intrinsic geometry of the moduli
space, and that a low-dimensional invariant projection discards information, since a
single ratio cannot separate the three-dimensional moduli space. Our algorithm instead
operates directly on the weighted variety, with a metric adapted to its grading.

Because \(d_F\) is a genuine metric (\cref{lem-7}), the single and average linkages
used here produce monotone dendrograms without inversions, and single linkage further
enjoys the stability guarantee of  \cref{thm-1}. Under the separation
hypotheses of \cref{lem-13} and \cref{lem-14}, the recovered clusters --- for instance
the families of curves with \((n,n)\)-split Jacobians --- reflect genuine proximity in
the moduli space rather than artifacts of a flat-space embedding. Absent such
separation, no method, ours included, guarantees recovery; the contribution is that the
metric makes the guarantee available precisely when the geometric conditions hold.


\section{Conclusion and Future Work}
We have constructed a scaling-invariant Finsler metric \( d_F([z], [w]) \) on weighted projective spaces \(\P_{\mathbf{q}}\), together with its rational counterpart \( d_{F,\Q}([z], [w]) \), and shown that both are genuine metrics (\cref{lem-7}, \cref{lem-8}).
This is the property the construction was built to secure: because \( d_F \)
satisfies the triangle inequality, the hierarchical clustering it induces admits
monotone dendrograms and is stable under perturbation of the input  (\cref{thm-1}),
a guarantee that the non-metric dissimilarities of \cite{2024-3} could not
provide. The metric thus supplies the theoretical footing for more structured
learning architectures on graded spaces, such as Graded Neural Networks, within
the broader program of \cite{2024-2, sh-95, sh-111}.

On the moduli side, clustering in \(\P_{(2,4,6,10)}(\Q)\) groups rational points
representing genus two curves by their Igusa invariants and, under the
separation hypotheses of \cref{lem-13,lem-14},
recovers families such as the loci of curves with \((n,n)\)-split Jacobians.
Where those geometric conditions hold, the recovered clusters reflect arithmetic
proximity in the moduli space rather than artifacts of a flat-space embedding;
where they fail, no method---ours included---can guarantee recovery. The
contribution is precisely that the metric makes the guarantee available when the
geometry permits it. The same machinery applies to arithmetic dynamics
\cite{2024-4} and to the reduction theory of binary forms \cite{2024-6},
illustrating the reach of geometry-aware learning over weighted varieties.
Several directions remain open. For \(p=2\), \cref{prop:closed-form} renders
distances in closed form, so the pressing computational question is the genuinely
Finsler regime \(p\neq 2\), where the variational and discrete-path schemes
establish feasibility but scaling them to high-dimensional datasets will require
dedicated solvers. A natural next step is to integrate
\( d_F \) directly into the loss functions of Graded Neural Networks
\cite{2025-5}, so that the cost of learning is weighted by the coordinate grades
rather than imposed by a flat metric.

A further direction lies in quantum computing, where weighted projective lines
can model the asymmetric noise profiles of NISQ devices. Treating circuit
landscapes as weighted manifolds would let Finsler geodesics steer optimization
away from decoherent regions, suggesting a route toward noise-resilient quantum
machine learning. Together these directions aim at a unified geometric
foundation for non-Euclidean data analysis across arithmetic geometry,
cryptography, and quantum information science.

 
 \bibliography{sh-97}

\begin{bibdiv}
\begin{biblist}

\bib{2024-4}{article}{
      author={Badr, E.},
      author={Shaska, E.},
      author={Shaska, T.},
       title={{Rational Functions on the Projective Line from a Computational
  Viewpoint}},
        date={2025},
      eprint={2503.10835},
         url={https://arxiv.org/abs/2503.10835},
}

\bib{BaoChernShen2000}{book}{
      author={Bao, David},
      author={Chern, Shiing-Shen},
      author={Shen, Zhongmin},
       title={{An Introduction to Riemann-Finsler Geometry}},
      series={Graduate Texts in Mathematics},
   publisher={Springer},
     address={New York},
        date={2000},
      volume={200},
        ISBN={978-0-387-98948-8},
}

\bib{quantum-brach}{article}{
      author={Bender, C.~M.},
      author={Brody, D.~C.},
      author={Jones, H.~F.},
      author={Meister, B.~K.},
       title={Faster than hermitian quantum mechanics},
        date={2007},
     journal={Physical Review Letters},
      volume={98},
       pages={040403},
}

\bib{Brody}{article}{
      author={Brody, D.~C.},
      author={Graefe, E.~M.},
       title={Information geometry of complex hamiltonians and quantum phase
  transitions},
        date={2013},
     journal={Physical Review Letters},
      volume={111},
       pages={230405},
}

\bib{quantum-wps}{article}{
      author={Brzezinski, T.},
      author={Fairfax, S.~A.},
       title={Quantum weighted projective spaces and weighted lens spaces},
        date={2012},
     journal={Journal of Geometry and Physics},
      volume={62},
       pages={43\ndash 55},
}

\bib{Dolgachev1982}{incollection}{
      author={Dolgachev, Igor},
       title={Weighted projective varieties},
        date={1982},
   booktitle={Group actions and vector fields},
      editor={Carrell, James~B.},
      series={Lecture Notes in Mathematics},
      volume={956},
   publisher={Springer},
     address={Berlin},
       pages={34\ndash 71},
}

\bib{Hastie2009}{book}{
      author={Hastie, Trevor},
      author={Tibshirani, Robert},
      author={Friedman, Jerome},
       title={The elements of statistical learning: Data mining, inference, and
  prediction},
     edition={2},
   publisher={Springer},
     address={New York},
        date={2009},
        ISBN={978-0-387-84857-0},
}

\bib{2024-6}{article}{
      author={Kotsireas, Ilias},
      author={Shaska, Tony},
       title={{A Neurosymbolic Framework for Geometric Reduction of Binary
  Forms}},
organization={American Mathematical Society},
        date={2026},
     journal={Contemporary Mathematics},
      volume={835},
       pages={173\ndash 193},
         url={https://doi.org/10.1090/conm/835/16757},
}

\bib{Zermelo}{article}{
      author={Russell, B.},
      author={Stepney, S.},
       title={Zermelo navigation and the quantum brachistochrone},
        date={2015},
     journal={Journal of Physics A: Mathematical and Theoretical},
      volume={48},
      number={40},
       pages={405303},
}

\bib{2022-1}{article}{
      author={Salami, Sajad},
      author={Shaska, Tony},
       title={Local and global heights on weighted projective varieties},
        date={2023},
        ISSN={0362-1588},
     journal={Houston J. Math.},
      volume={49},
      number={3},
       pages={603\ndash 636},
      review={\MR{4845203}},
}

\bib{santos2014}{thesis}{
      author={Santos, M'{a}rcio~Silva},
       title={On the geometry of weighted manifolds},
        type={Ph.D. Thesis},
     address={Macei'{o}, Brazil},
        date={2014},
        note={Tese de Doutorado},
}

\bib{2025-6}{article}{
      author={Shaska, E.},
      author={Shaska, T.},
       title={Weighted heights and {GIT} heights},
        date={2026},
        ISSN={0925-9899,1572-9192},
     journal={J. Algebraic Combin.},
      volume={63},
      number={3},
       pages={Paper No. 41, 31},
  url={https://doi-org.huaryu.kl.oakland.edu/10.1007/s10801-026-01534-7},
      review={\MR{5065331}},
}

\bib{2024-3}{article}{
      author={Shaska, Elira},
      author={Shaska, Tanush},
       title={Machine learning for moduli space of genus two curves and an
  application to isogeny-based cryptography},
        date={2025},
        ISSN={0925-9899,1572-9192},
     journal={J. Algebraic Combin.},
      volume={61},
      number={2},
       pages={Paper No. 23, 35},
         url={https://doi.org/10.1007/s10801-025-01393-8},
      review={\MR{4870337}},
}

\bib{2025-5}{article}{
      author={Shaska, Tony},
       title={{Graded Neural Networks}},
        date={2025},
        ISSN={2810-9392,2810-9406},
     journal={Int. J. Data Sci. Math. Sci.},
      volume={3},
      number={2},
       pages={87\ndash 116},
         url={https://doi-org.huaryu.kl.oakland.edu/10.1142/s2810939225500030},
      review={\MR{5061522}},
}

\bib{sh-111}{article}{
      author={Shaska, Tony},
       title={{Internalizing Tools as Morphisms in Graded Transformers}},
        date={2025},
      eprint={2511.17840},
         url={https://arxiv.org/abs/2511.17840},
}

\bib{2024-2}{article}{
      author={Shaska, Tony},
       title={{Artificial Neural Networks on Graded Vector Spaces}},
        date={2026},
     journal={Contemp. Math.},
      volume={835},
       pages={1\ndash 71},
         url={https://doi-org.huaryu.kl.oakland.edu/10.1090/conm/835/16752},
      review={\MR{5023824}},
}

\bib{sh-95}{article}{
      author={Shaska, Tony},
       title={{Graded Transformers}},
        date={2026},
        ISSN={2810-9392,2810-9406},
     journal={Int. J. Data Sci. Math. Sci.},
      volume={accepted},
}

\bib{Shen2012}{book}{
      author={Shen, Zhongmin},
       title={{Lectures on Finsler Geometry}},
   publisher={World Scientific},
     address={Singapore},
        date={2012},
        ISBN={978-981-3102-61-3},
}

\end{biblist}
\end{bibdiv}

\end{document}